\documentclass{amsart}
\usepackage[sc]{mathpazo}

\usepackage[margin=1in]{geometry}
\usepackage{amssymb}
\usepackage{hyperref}
\usepackage{shuffle}
\usepackage{amsmath,amscd}
\usepackage[enableskew]{youngtab}
\usepackage[all]{xy}
\usepackage{graphicx}

\newtheorem{theorem}{Theorem}[section]
\newtheorem{lemma}[theorem]{Lemma}
\newtheorem{proposition}[theorem]{Proposition}
\newtheorem{corollary}[theorem]{Corollary}

\theoremstyle{definition}
\newtheorem{definition}[theorem]{Definition}
\newtheorem{example}[theorem]{Example}

\theoremstyle{remark}

\numberwithin{equation}{section}

\def\qand{\quad\text{and}\quad}
\DeclareMathOperator{\op}{op}

\DeclareMathOperator{\soc}{soc}
\DeclareMathOperator{\Hom}{Hom}

\DeclareMathOperator{\Sym}{Sym}
\DeclareMathOperator{\QSym}{QSym}
\DeclareMathOperator{\NSym}{\mathbf{NSym}}

\let\top\relax \DeclareMathOperator{\top}{top}
\def\bS{\mathbf{S}}
\DeclareMathOperator{\Irr}{Irr}  
 \def\CC{{\mathbb C}} \def\FF{{\mathbb F}}

\def\H{\mathcal{H}}  \def\SS{\mathfrak{S}}
\def\bP{\mathbf{P}} \def\bC{\mathbf{C}} 
\def\bM{\mathbf{M}} \def\bN{\mathbf{N}} \def\bQ{\mathbf{Q}}
\def\bq{\mathbf{q}}

\def\htimes{\,\widehat\otimes\,}
\def\pib{\overline{\pi}}  \def\rad{\operatorname{rad}}

\def\J{\mathcal{J}}

\begin{document}

\title{Hecke algebras of simply-laced type with independent parameters}
\author{Jia Huang}
\address{Department of Mathematics and Statistics, University of Nebraska at Kearney, Kearney, Nebraska, USA}
\curraddr{}
\email{huangj2@unk.edu}
\thanks{
}
\subjclass[2010]{16G30, 05E10}

\keywords{Hecke algebra, independent parameters, simply-laced Coxeter system, induction and restriction, duality}

\maketitle

\begin{abstract}
We study the (complex) Hecke algebra $\mathcal{H}_S(\mathbf{q})$ of a finite simply-laced Coxeter system $(W,S)$ with independent parameters $\mathbf{q} \in \left( \mathbb{C} \setminus\{\text{roots of unity}\} \right)^S$. We construct its irreducible representations and projective indecomposable representations. We obtain the quiver of this algebra and determine when it is of finite representation type. We provide decomposition formulas for induced and restricted representations between the algebra $\mathcal{H}_S(\mathbf{q})$ and the algebra $\mathcal{H}_R(\mathbf{q}|_R)$ with $R\subseteq S$. Our results demonstrate an interesting combination of the representation theory of finite Coxeter groups and their 0-Hecke algebras, including a two-sided duality between the induced and restricted representations.
\end{abstract}

\section{Introduction}

Let $W:= \left\langle\, S:(st)^{m_{st}}=1,\ \forall s,t\in S\,\right\rangle$ be a Coxeter group generated by a finite set $S$ with relations $(st)^{m_{st}}=1$ for all $s,t\in S$, where $m_{ss}=1$ for all $s\in S$ and $m_{st}=m_{ts}\in\{2,3,\ldots\}\cup\{\infty\}$ for all distinct $s,t\in S$. 
Given a parameter $q$ in a field $\FF$, the \emph{(Iwahori-)Hecke algebra} $\H_S(q)$ of the Coxeter system $(W,S)$ is the (unital associative) algebra over $\FF$ generated by $\{T_s:s\in S\}$ with 
\begin{itemize}
\item
quadratic relations $(T_s-1)(T_s+q)=0$ for all $s\in S$, and
\item
braid relations $(T_sT_tT_s\cdots)_{m_{st}} = (T_tT_sT_t\cdots)_{m_{st}}$ for all $s,t\in S$.
\end{itemize}
Here $(aba\cdots)_m$ denotes an alternating product of $m$ terms.
The Hecke algebra $\H_S(q)$ is a one-parameter deformation of the group algebra $\FF W$ of $W$.
It has an $\FF$-basis $\{T_w:w\in W\}$ indexed by $W$, where $T_w:=T_{s_1}\cdots T_{s_\ell}$ if $w=s_1\cdots s_\ell$ is a reduced (i.e., shortest) expression in the generators of $W$.
The Hecke algebra $\H_S(q)$ naturally arises in different ways and has significance in many areas (see, e.g., Lusztig~\cite{Lusztig03}).

Tits showed that, if $W$ is finite, $\FF=\CC$ is the field of complex numbers, and $q\in\CC$ is neither zero nor a root of unity, then the Hecke algebra $\H_S(q)$ is semisimple and isomorphic to the group algebra $\CC W$.
The representation theory of Hecke algebras at roots of unity has been studied to some extent, with connections to other topics found (see Geck and Jacon~\cite{HeckeRoot}), but has not been completely determined yet even in type $A$ (see Goodman and Wenzl~\cite{HeckeRootA}).

Another interesting specialization of $\H_S(q)$ is the \emph{$0$-Hecke algebra} $\H_S(0)$, which is different from but closely related to the group algebra of $W$.
It was used by Stembridge~\cite{Stembridge} to give a short derivation for the M\"obius function of the Bruhat order of the Coxeter group $W$ and its parabolic quotients. 

For a finite Coxeter system $(W,S)$, Norton~\cite{Norton} studied the representation theory of $\H_S(0)$ over an arbitrary field $\FF$ using the triangularity of the product in $\H_S(0)$.
Her main result is a decomposition of $\H_S(0)$ into a direct sum of $2^{|S|}$ many indecomposable submodules; 
this decomposition is similar to the decomposition of the group algebra of $W$ (over the field of rational numbers) by Solomon~\cite{Solomon}.
Norton's result provided motivations to later work of Denton, Hivert, Schilling, and Thi\'ery~\cite{J} on the representation theory of finite \emph{$J$-trivial monoids}, as $\H_S(0)$ is a monoid algebra of the \emph{0-Hecke monoid} $\left\{ T_w|_{q=0} : w\in W \right\}$, an example of $J$-trivial monoids.

In type $A$, Krob and Thibon~\cite{KrobThibon} discovered an important correspondence from representations of 0-Hecke algebras to \emph{quasisymmetric functions and noncommutative symmetric functions}.
This correspondence is analogous to the classical Frobenius correspondence from complex representations of symmetric groups to symmetric functions.
Duchamp, Hivert, and Thibon~\cite{NCSF-VI} constructed the quiver of the $0$-Hecke algebra $H_n(0)$ of type $A_{n-1}$ and showed that $H_n(0)$ is of infinite representation type for $n\ge4$.
Tewari and van Willigenburg~\cite{TvW15,TvW19} and K\"onig~\cite{Konig} studied connections between $H_n(0)$ and a new basis of quasisymmetric functions called the \emph{quasisymmetric Schur functions}. 
After further investigation of combinatorial aspects of the representation theory of $H_n(0)$~\cite{H0CF,H0SR} using the correspondence of Krob and Thibon, we recently extended this correspondence from type $A$ to type $B$ and type $D$~\cite{H0Tab,H0Ch}.

Motivated by the similarities and differences between various specializations of the Hecke algebra $H_S(q)$, we generalized its definition from a single parameter $q$ to multiple independent parameters and studied the resulting algebra in recent work~\cite{Hecke}.
The \emph{Hecke algebra $\H_S(\mathbf q)$ of a Coxeter system $(W,S)$ with independent parameters $\bq=(q_s\in\FF:s\in S)\in\FF^S$} is the $\FF$-algebra generated by $\{T_s:s\in S\}$ with 
\begin{itemize}
\item
quadratic relations $(T_s-1)(T_s+q_s)=0$ for all $s\in S$, and
\item
braid relations $(T_sT_tT_s\cdots)_{m_{st}} = (T_tT_sT_t\cdots)_{m_{st}}$ for all $s,t\in S$.
\end{itemize}
We constructed a basis for $\H_S(\bq)$ when $(W,S)$ is simply laced and characterized when $\H_S(\bq)$ is commutative.
In type $A$, a commutative Hecke algebra $\H_S(\bq)$ has its dimension given by a Fibonacci number and its representation theory has interesting features analogous to the representation theory of both symmetric groups and their $0$-Hecke algebras.

In this paper we investigate the representation theory of the (not necessarily commutative) algebra $\H_S(\bq)$ when $(W,S)$ is a simply-laced Coxeter system. 
Our result shows an interesting combination of the representation theory of Coxeter groups and $0$-Hecke algebras, but there are certain features of the representation theory of $\H_S(\bq)$ that are unlike both symmetric groups and $0$-Hecke algebras.
For example, restrictions of projective $\H_S(\bq)$-modules are not projective in general.

We do not consider roots of unity here, but allowing these parameters would still give interesting algebras whose representation theory is yet to be determined.
Although we focus on simply-laced Coxeter systems, some of the preliminary results in Section~\ref{sec:prelim} are valid for all finite Coxeter systems, and it may be possible to extend our results to non-simply-laced Coxeter systems.

This paper is structured as follows. 
In Section~\ref{sec:prelim} we review the representation theory of finite dimensional algebras, finite Coxeter groups, and $0$-Hecke algebras, and also develop some basic results for later use.
Next, in Section~\ref{sec:dim} we study the structure of the algebra $\H_S(\bq)$ and give a formula for its dimension.
In Section~\ref{sec:decomp} we construct the projective indecomposable $\H_S(\bq)$-modules and simple $\H_S(\bq)$-modules, and determine the Cartan matrix and (ordinary) quiver of $\H_S(\bq)$.
In Section~\ref{sec:IndRes} we obtain formulas for induction and restriction of representations between $\H_R(\bq)$ and $\H_S(\bq|_R)$ for $R\subseteq S$, and verify a two-sided duality between induction and restriction.
Lastly, we give some remarks and questions in Section~\ref{sec:remark}.



\section{Preliminaries}\label{sec:prelim}

\subsection{Representations of algebras}
We first review some general results on representations of algebras; see references~\cite{ASS,CurtisReiner,Etingof} for more details.
All algebras and modules in this paper are \emph{finite dimensional}.

Let $A$ be an (unital associative) algebra over a field $\FF$. 
We say $A$ is of \emph{finite [or infinite, resp.] representation type} if the number of non-isomorphic indecomposable $A$-modules is finite [or infinite, resp.].
Let $M$ be a (left) $A$-module. 
The \emph{radical} $\rad(M)=\rad_A(M)$ of $M$ is the intersection of all maximal $A$-submodules of $M$.
We have $\rad(M)=\rad(A)M$, where $A$ is treated as an $A$-module. 
Let $\rad^{i+1}(M):=\rad(\rad^i(M))$ for $i=1,2,\ldots$.
The \emph{top} of $M$ is $\top(M)=\top_A(M) := M/\rad_A(M)$, which is the largest semisimple quotient of $M$.
The \emph{socle} $\soc(M)=\soc_A(M)$ of $M$ is the sum of all simple submodules of $M$, which is the largest semisimple submodule of $M$.

There exists a direct sum decomposition $A=\bP_1\oplus\cdots\oplus\bP_k$ where $\bP_1,\ldots,\bP_k$ are indecomposable $A$-submodules.
For $i=1,2,\ldots,k$, the radical $\rad(\bP_i)$ is the unique maximal $A$-submodule of $\bP_i$ and $\bC_i:=\top(\bP_i)$ is simple~\cite[Proposition~I.4.5~(c)]{ASS}.
Moreover, every projective indecomposable $A$-module is isomorphic to $\bP_i$ for some $i$ and every simple $A$-module is isomorphic to $\bC_j$ for some $j$.

The algebra $A$ is \emph{basic} if $\bP_1,\ldots,\bP_k$ are pairwise non-isomorphic.
In general, we may assume, without loss of generality, that $\{\bP_1,\ldots,\bP_r\}$ is a complete set of pairwise non-isomorphic projective $A$-modules and $\{\bC_1,\ldots,\bC_r\}$ is a complete set of pairwise non-isomorphic simple $A$-modules for some $r\le k$. 
The \emph{Cartan matrix} of $A$ is $[c_{ij}]_{i,j=1}^r$ where $c_{ij}:=\dim_\FF \Hom_A(\bP_i,\bP_j)$ is the multiplicity of the simple module $\bC_i=\top(\bP_i)$ among the composition factors of the projective indecomposable module $\bP_j$.

The \emph{Grothendieck groups} $G_0(A)$ and $K_0(A)$ of $A$ are free abelian groups with bases $\{\bC_1,\ldots,\bC_r\}$ and $\{\bP_1,\ldots,\bP_r\}$, respectively.
If $0\to L\to M\to N\to 0$ is a short exact sequence of $A$-modules [or projective $A$-modules, resp.], then $M$ is identified with $L+N$ in $G_0(A)$ [or $K_0(A)$, resp.].
If $A$ is a semisimple algebra then $G_0(A)=K_0(A)$.

If $M$ and $N$ are two $A$-modules then define $\langle M,N\rangle:=\dim_\FF \Hom_A(M,N)$.
Since $\langle \bC_i,\bC_j\rangle =  \delta_{i,j}$ for all $i$ and $j$ by Schur's Lemma and since $f(\rad(M))\subseteq \rad(N)$ for any $f\in\Hom_A(M,N)$, we have
\begin{equation}\label{eq:hom}
\langle M,\bC_j\rangle = \langle \bC_i,\bC_j\rangle =  \delta_{i,j} \quad 
\text{if $\top(M)\cong \bC_i$}
\end{equation}
where  $\delta$ is the Kronecker delta.
Taking $M=\bP_i$ gives the duality between $G_0(A)$ and $K_0(A)$. 

We next provide some basic results on representations of algebras for later use.

\begin{proposition}\label{prop:tensor}
Let $A$ and $B$ be two algebras. Let $M$ be an $A$-module and $N$ a $B$-module. Then the following statements hold for the $A\otimes B$-module $M\otimes N$.
\begin{itemize}
\item[(i)]
We have $\rad(M\otimes N) \cong \rad(M)\otimes N + M \otimes\rad(N)$ and $\top(M\otimes N) \cong \top(M) \otimes \top(N)$.
\item[(ii)]
If $M$ and $N$ are both simple then $M\otimes N$ is also simple. 
Conversely, any simple $A\otimes B$-module can be written as $M\otimes N$ for a unique $A$-module $M$ and a unique $B$-module $N$.
\item[(iii)]
We have $\rad(M\otimes N)/\rad^2(M\otimes N) = \left( \rad(M)/\rad^2(M) \right) \otimes \top(N) + \top(M) \otimes \left( \rad(N) / \rad^2 (N) \right)$.
\end{itemize}
\end{proposition}

\begin{proof}
Part (i) and Part (ii) follow from a standard result~\cite[Theorem 2.26]{Etingof} and its proof.
Applying (i) gives $\rad^2(M\otimes N) = \rad^2(M) \otimes N + \rad(M) \otimes \rad(N) + M \otimes \rad^2(N)$, which implies (iii).
\end{proof}


\begin{proposition}\label{prop:TensorP}
Let $A = \bP_1 \oplus\cdots\oplus \bP_k$ and $A' = \bP'_1 \oplus \cdots\oplus \bP'_\ell$ be direct sum decompositions of two algebras $A$ and $A'$ into indecomposable submodules. Then 
\[ A\otimes A' = \bigoplus_{i=1}^k \bigoplus_{j=1}^{\ell } \left( \bP_i\otimes\bP'_j \right) \]
where each summand $\bP_i\otimes\bP'_j$ is an indecomposable $A\otimes A'$-module with $\top(\bP_i\otimes\bP'_j) \cong \top(\bP_i) \otimes \top(\bP'_j)$.
\end{proposition}

\begin{proof}
By the distributivity of tensor product over direct sum, $A\otimes A'$ is a direct sum of $\bP_i\otimes \bP'_j$ for  all $i=1,\ldots,k$ and $j=1,\ldots,\ell$.
By Proposition~\ref{prop:tensor}, $\top(\bP_i\otimes\bP'_j) = \top(\bP_i) \otimes \top(\bP_j)$ is a simple $A\otimes A'$-module.
Hence $\bP_i\otimes\bP'_j$ must be indecomposable.
\end{proof}

Now suppose that there is an algebra homomorphism $\phi:A\to B$.
Any $B$-module $M$ becomes an $A$-module by $am:=\phi(a)m$, $\forall a\in A$, $\forall m\in M$.
We call this $A$-module the \emph{restriction} of $M$ from $B$ to $A$ and denote it by $M\downarrow\,_A^B$.
The \emph{induction} of an $A$-module $N$ from $A$ to $B$ is the $B$-module $N\uparrow\,_A^B := B\otimes_A N$, where $B=\,\!_BB_A$ is regarded as a $(B,A)$-bimodule.

\begin{proposition}\label{prop:induce}
Suppose that $\phi:A\to B$ is an algebra homomorphism and $M$ is a $B$-module.
\begin{itemize}
\item[(i)]
If $M\downarrow\,_A^B$ is a simple [or indecomposable, resp.] $A$-module then $M$ is a simple [or indecomposable, resp.] $B$-module. 
\item[(ii)]
If $M$ is projective indecomposable both as an $A$-module and as a $B$-module, and if $\rad_A(M)$ is a $B$-submodule of $M$, then $\rad_A(M) = \rad_B(M)$.
\end{itemize}
\end{proposition}

\begin{proof}
Any $B$-submodule of $M$ restricts to an $A$-module.
This implies (i).

If $M$ is a projective indecomposable $A$-module then $\rad_A(M)$ is the unique maximal $A$-submodule of $M$~\cite[Proposition~I.4.5~(c)]{ASS} and the same result holds for $\rad_B(M)$ if $M$ is a projective $B$-module.
Since $\rad_B(M)$ restricts to a proper $A$-submodule of $M$, it is contained in $\rad_A(M)$.
On the other hand, if $\rad_A(M)$ is a $B$-module then it is contained in $\rad_B(M)$.
Thus (ii) holds.
\end{proof}

The proof of the following proposition is left to the reader as an exercise.
\begin{proposition}\label{prop:surj}
Let $A\twoheadrightarrow B$ be a surjection of algebras.
\begin{enumerate}
\item[(i)]
A $B$-module is simple (or indecomposable) if and only if its restriction to $A$ is simple (or indecomposable).
\item[(ii)]
Two $B$-modules are isomorphic if and only if their restrictions to $A$ are isomorphic.
\item[(iii)]
If $A$ is of finite representation type then so is $B$.
\end{enumerate}
\end{proposition}

Under certain circumstances, e.g., when $A$ and $B$ are group algebras over the complex field $\CC$, one has 
\begin{equation}\label{eq:dual1}
\Hom_B(N\uparrow\,_A^B,M) \cong \Hom_A(N,M\downarrow\,_A^B).
\end{equation}
This is known as the \emph{Frobenius reciprocity}. 
The other possible adjunction 
\begin{equation}\label{eq:dual2}
\Hom_B(N\downarrow\,_A^B,M) \cong \Hom_A(N,M\uparrow\,_A^B)
\end{equation}
holds for the (complex) group algebras of the symmetric groups and their $0$-Hecke algebras (over any field $\FF$), giving the duality between certain graded Hopf algebras (see Section~\ref{sec:TypeA}).

Next, recall that a \emph{quiver} $Q$ is a directed graph possibly with loops and multiple arrows between two vertices.
Its \emph{path algebra} $\CC Q$ has a basis consisting of all paths in $Q$ and has multiplication given by concatenation of paths.
The \emph{arrow ideal} $R_Q$ is the two-sided ideal of $\CC Q$ generated by all arrows in $Q$.
A \emph{representation} of $Q$ is a $\CC Q$-module.
Gabriel's theorem classifies connected quivers of finite representation type as type $A_n$, $D_n$, $E_6$, $E_7$, and $E_8$, meaning that these quivers do not contain oriented cycles and their underlying undirected graphs are given by Coxeter diagrams of the corresponding types (cf. Section~\ref{sec:Coxeter}).

Let $A$ be a finite dimensional $\CC$-algebra whose projective indecomposable modules are $\bP_1,\ldots,\bP_r$ and let $\bC_i:=\top(\bP_i)$ for all $i$. 
The \emph{(ordinary) quiver} $Q_A$ of $A$ is the direct graph with vertices $\bC_1,\ldots,\bC_r$ such that the number of arrows from $\bC_i$ to $\bC_j$ is the multiplicity of $\bC_j$ among the composition factors of $\rad(\bP_i)/\rad^2(\bP_i)$.
In particular, the quiver of a semisimple algebra $A$ consists of isolated vertices.

If $A$ is a basic algebra then there exists an ideal $I$ of the path algebra $\CC Q_A$ such that $A\cong \CC Q_A/I$ and $I\subseteq R^2$, where $R$ is the arrow ideal of $Q_A$~\cite[Theorem~II.3.7]{ASS}.
If $A$ is not basic then there is a basic algebra $A^b$ such that the categories of finitely generated modules over $A$ and $A^b$ are equivalent~\cite[Corollary~I.6.10]{ASS} and the quiver of $A$ is the same as the quiver of $A^b$ (cf. Li and Chen~\cite[Proposition~1.2]{LiChen}).

Assume $A_1$ and $A_2$ are two algebras whose quivers $Q_1$ and $Q_2$ are loopless.
The quiver of $A_1\otimes A_2$ is the \emph{tensor product} $Q_1\otimes Q_2$ of $Q_1$ and $Q_2$, a loopless quiver defined below: its vertex set is the Cartesian product of the vertex sets of $Q_1$ and $Q_2$, and the number of arrows from $(u_1,u_2)$ to $(v_1,v_2)$ is
\[ \begin{cases}
\text{the number of arrows from $u_1$ to $v_1$}, & \text{if $u_1\ne v_1$ and $u_2=v_2$},\\
\text{the number of arrows from $u_2$ to $v_2$}, & \text{if $u_1=v_1$ and $u_2\ne v_2$},\\
\text{zero}, & \text{otherwise}.
\end{cases}\]

\subsection{Coxeter groups and their representation theory}\label{sec:Coxeter}
We recall some basic definitions and results on Coxeter groups from Bj\"orner and Brenti~\cite{BjornerBrenti}.
A \emph{Coxeter group} is a group $W$ generated by a finite set $S$ with quadratic relations $s^2=1$ for all $s\in S$ and braid relations $(sts\cdots)_{m_{st}} = (tst\cdots)_{m_{st}}$ for all distinct $s,t\in S$, where $m_{st}=m_{ts}\in\{2,3,\ldots\}\cup\{\infty\}$ and $(aba\cdots)_m$ denotes an alternating product of $m$ terms.
The \emph{Coxeter system} $(W,S)$ can be encoded by an edge-labeled graph called the \emph{Coxeter diagram} of $(W,S)$; the vertex set of this graph is $S$ and there is an edge labeled $m_{st}$ between distinct vertices $s$ and $t$ whenever $m_{st}\ge 3$.
If $m_{st}\le 3$ for all distinct $s,t\in S$ then $(W,S)$ is \emph{simply laced}.
We say $(W,S)$ is \emph{finite} if $W$ is finite.
If the Coxeter diagram of $(W,S)$ is connected then $(W,S)$ is \emph{irreducible}. 
There is a well-known classification of finite irreducible Coxeter systems as type $A_n$, $B_n$, $D_n$, $I_2(m)$, $E_6$, $E_7$, $E_8$, $F_4$, $H_3$, $H_4$~\cite[Appendix A1]{BjornerBrenti}.

Let $(W,S)$ be a Coxeter system and let $w\in W$.
We say that $w=s_1\cdots s_k$ is a \emph{reduced expression} of $w$ if $s_1,\cdots,s_k\in S$ and $k$ is as small as possible; the smallest $k$ is the \emph{length} $\ell(w)$ of $w$.
The \emph{descent set} of $w$ is defined as $D(w) := \{w\in S: \ell(ws)<\ell(w)\}$ and its elements are called the \emph{descents} of $w$.
One has $s\in D(w)$ if and only if some reduced expression of $w$ ends with $s$.

Given $I\subseteq S$, the \emph{parabolic subgroup} $W_I$ of $W$ is generated by $I$. 
The pair $(W_I,I)$ is a Coxeter system whose Coxeter diagram has vertex set $I$ and has labeled edges $(s,t)$ of the Coxeter diagram of $(W,S)$ for all $s,t\in I$.
Each left coset of $W_I$ in $W$ has a unique minimal representative.
The set of all minimal representatives of left $W_I$-cosets is 
$W^I:=\{w\in W:D(w)\subseteq S\setminus I\}.$
Every element of $W$ can be written uniquely as $w=w^I\cdot\,_Iw$, where $w^I\in W^I$ and $_Iw\in W_I$;
this implies $\ell(w)=\ell(w^I)+\ell(_Iw)$.

Let $I\subseteq S$.
The \emph{descent class} of $I$ in $W$ is $\{w\in W: D(w)=I\}$.
When $W$ is finite, the descent class of $I$ is nonempty by Lusztig~\cite[Lemma~9.8]{Lusztig03} and becomes an interval $[w_0(I),w_1(I)]$ under the left weak order of $W$ by Bj\"orner and Wachs~\cite[Theorem~6.2]{BjornerWachs}.
Here $w_0(I)$ and $w_1(I)$ are the longest elements of $W_I$ and $W^{S\setminus I}$, respectively, and the \emph{left weak order} is a partial ordering on $W$ defined by setting $u\le_L w$ if there exists some reduced expression $w=s_1\cdots s_k$ such that $s_i\cdots s_k=u$ for some $i$. 

Another important partial order on $W$ is the \emph{Bruhat order}: given $u,w\in W$, define $u\le w$ if a reduced expression of $u$ is a subword of some (or equivalently, every) reduced expression of $w$.
When $W$ is finite, its longest element $w_0$ is the unique maximum element in Bruhat order and can be characterized by the property $\ell(sw_0)<\ell(w_0)$ for all $s\in S$~\cite[Prop.~2.3.1]{BjornerBrenti}.

An important example of a finite Coxeter group is the symmetric group $\SS_n$, and we will review its basic properties in Section~\ref{sec:TypeA}.
The representation theory of $\SS_n$ is well studied and can be extended to finite Coxeter groups of other types (see, e.g., Adin--Brenti--Roichman~\cite{ABR1,ABR2} and Humphreys~\cite[\S8.10]{Humphreys}).
With that in mind, we adopt some notation below for the complex representation theory of a finite group.

The group algebra $\CC G$ is semisimple and every $\CC G$-module is a direct sum of simple/irreducible $\CC G$-submodules.
There exists a complete list $\{\bS_\lambda: \lambda \in \Irr(\CC G) \}$ of pair-wise nonisomorphic simple $\CC G$-modules, where $\Irr(\CC G)$ is in bijection with the set of conjugacy classes of $G$.
By Schur's Lemma, the Cartan matrix of $\CC G$ is the identity matrix $[\delta_{\lambda,\mu}]_{\lambda,\mu\in\Irr(\CC G)}$, where $\delta$ is the Kronecker delta.
The span of $\sigma(G) := \sum_{g\in G} g$ is the \emph{trivial representation} of $G$, whose complement in $\CC G$ is spanned by the set
\begin{equation}\label{eq:perp}
\sigma(G)^\perp := \left\{ \sum_{g\in G} c_g g : c_g\in \CC, \ \sum_{g\in G} c_g = 0 \right\}. 
\end{equation}
The regular representation of $G$ has a decomposition
\begin{equation}\label{eq:DecompCG}
\CC G = \CC \sigma(G) \oplus \CC \sigma(G)^\perp \cong \bigoplus_{\lambda\in \Irr(\CC G)} (\bS_\lambda)^{\oplus d_\lambda}.
\end{equation}
Here $d_\lambda$ be the dimension of $\bS_\lambda$ for each $\lambda\in \Irr(\CC G)$; in particular, $d_\lambda=1$ if $\bS_\lambda \cong \CC\sigma(G)$ is trivial.

Let $H$ be a subgroup of $G$.
There exists an integer $c_\mu^\lambda\ge0$ for all $\lambda\in\Irr(\CC G)$ and $\mu\in\Irr(\CC H)$ such that
\begin{equation}\label{eq:IndResG}
\bS_\mu\uparrow\,_H^G \cong \bigoplus_{\lambda\in\Irr(\CC G)} \bS_\lambda^{c_\mu^\lambda} \qand
\bS_\lambda\downarrow\,_H^G \cong \bigoplus_{\mu\in\Irr(\CC H)} \bS_\mu^{c_\mu^\lambda}. 
\end{equation}
Thus the \emph{Frobenius Reciprocity} holds: if $\lambda\in\Irr(\CC G)$ and $\mu\in\Irr(\CC H)$ then
\[ \left\langle \bS_\lambda, \bS_\mu\uparrow\,_H^G \right\rangle 
= \left\langle \bS_\lambda\downarrow\,_H^G, \bS_\mu \right\rangle. \]

The above restriction formula~\eqref{eq:IndResG} implies the following lemma, which will be useful in Section~\ref{sec:IndRes}.

\begin{lemma}\label{lem:RestrictOne}
Let $H$ be a subgroup of $G$. 
If $\bS_\lambda$ is trivial, where $\lambda\in\Irr(\CC G)$, and $c_\mu^\lambda\ne0$ for some $\mu\in\Irr(\CC H)$, then $\bS_\mu$ is also trivial.
\end{lemma}

\begin{proof}
Since $G$ acts trivially on $\bS_\lambda$, so does its subgroup $H$.
Thus $c_\mu^\lambda\ne0$ implies $\bS_\mu$ is trivial.
\end{proof}

\subsection{0-Hecke algebras}\label{sec:H0}
Now we recall the definition and properties of the 0-Hecke algebras; see, e.g., Krob--Thibon~\cite{KrobThibon}, Norton~\cite{Norton}, and Stembridge~\cite{Stembridge}.
The \emph{$0$-Hecke algebra} $\H_S(0)$ of a Coxeter system $(W,S)$ over an arbitrary field $\FF$ is the specialization of the Hecke algebra $\H_S(q)$ of $(W,S)$ at $q=0$, i.e, the $\FF$-algebra generated by $\{\pi_s:s\in S\}$ with quadratic relations $\pi_s^2=\pi_s$ for all $s\in S$ and braid relations $(\pi_s\pi_t\pi_s\cdots)_{m_{st}} = (\pi_t\pi_s\pi_t\cdots)_{m_{st}}$ for all distinct $s,t\in S$, where $\pi_s:=T_s|_{q=0}$.
There is another generating set $\{\pib_s: s\in S\}$ for $\H_S(0)$, where $\pib_s:=\pi_s-1$ (so that $\pi_s\pib_s=\pib_s\pi_s=0$), with quadratic relations $\pib_s^2=-\pib_s$ for all $s\in S$ and the same braid relations as above. 

There are two $\FF$-bases $\{\pi_w:w\in W\}$ and $\{\pib_w:w\in W\}$ for $\H_S(0)$, where $\pi_w:=\pi_{s_1}\cdots\pi_{s_k}$ and $\pib_w:=\pib_{s_1}\cdots\pib_{s_k}$ for any reduced expression $w=s_1\cdots s_k$.
For each $w\in W$ we have
\begin{equation}\label{eq:pib}
\pi_w =\sum_{u\le w} \pib_u \qand \pib_w = \sum_{u\le w} (-1)^{\ell(w)-\ell(u)} \pi_u
\end{equation}
where ``$\le$'' is the Bruhat order of $W$.
\footnote{The two equalities in \eqref{eq:pib} are equivalent to each other by the automorphism $\pi_i\mapsto -\pib_i$ of the algebra $\H_S(0)$.
This gives the short derivation for the M\"obius function of the Bruhat order of $W$  by Stembridge~\cite{Stembridge}.}
If $s\in S$ and $w\in W$ then 
\begin{equation}\label{eq:H0Prod}
\pi_s \pi_w = \begin{cases}
\pi_{sw}, & \text{ if } \ell(sw)>\ell(w), \\
\pi_w, & \text{ otherwise}, 
\end{cases} \qand
\pib_s \pib_w = \begin{cases}
\pib_{sw}, & \text{ if } \ell(sw)>\ell(w), \\
-\pib_w, & \text{ otherwise}.
\end{cases} 
\end{equation}

Assume the Coxeter system $(W,S)$ is finite below.
Norton~\cite{Norton} obtained a decomposition
\begin{equation}\label{eq:H0Decomp}
\H_S(0) = \bigoplus_{I\subseteq S} \bP_I^S 
\end{equation}
where $\bP_I^S := \H_S(0) \pi_{w_0(I)} \pib_{w_0(S\setminus I)}$ is an indecomposable submodule of $\H_S(0)$ with an $\FF$-basis 
\[ \left\{\pi_w\pib_{w_0(S\setminus I)} : w\in W,\ D(w)=I \right\}.\]
If $s\in S$ and $w\in W$ then by the multiplication rule~\eqref{eq:H0Prod} and the relation $\pi_t\pib_t=0$ for any $t\in S$, we have
\begin{equation}\label{eq:H0Action}
\pi_s \pi_w\pib_{w_0(S\setminus I)} = 
\begin{cases}
\pi_w\pib_{w_0(S\setminus I)}, & \text{if } s\in D(w^{-1}), \\
0, & \text{if } s\notin D(w^{-1}),\ D(s_iw) \ne I, \\
\pi_{sw}\pib_{w_0(S\setminus I)}, & \text{if } s\notin D(w^{-1}),\ D(s_iw) = I
\end{cases} 
\end{equation}
where $I:=D(w)$.
The top $\bC_I^S$ of $\bP_I^S$ is a one-dimensional simple $\H_S(0)$-module on which $\pi_s$ acts by $1$ if $s\in I$ or by $0$ if $s\in S\setminus I$.
The socle of $\bP_I^S$ is a one-dimensional simple module generated by $\pi_{w_1(I)}\pib_{w_0(S\setminus I)}$.

Every cyclic $\H_S(0)$-module $\H_S(0) v$ admits a length filtration 
\[ 0=\H_S^k(0)v\subseteq\H_S^{k-1}(0)v\subseteq \cdots\subseteq \H_S^0(0)v =\H_S(0)v\]
for some positive integer $k$, where $\H_S^i(0)$ is the span of $\{ \pi_w: w\in W,\ \ell(w)\ge i\}$ for all $i=0,1,\ldots,k$.
Given $I,J\subseteq S$, refining the above filtration to a composition series for the cyclic module $\bP_J^S(0)$ gives
\[ c_{I,J}^S := \dim_\FF \Hom_{\H_S(0)} \left( \bP_I^S,\bP_J^S \right) = \# \left\{w\in W: D(w^{-1}) = I,\ D(w)= J \right\} \]
by the equation~\eqref{eq:H0Action}. 
Thus the Cartan matrix of $\H_S(0)$ is the symmetric matrix $\left[ c_{I,J}^S \right]_{I,J\subseteq S}$.

We next study certain quotients of projective indecomposable $\H_S(0)$-modules, which will help with our study of restricted representations in Section~\ref{sec:IndRes}.
Examples in type $A$ are given by Figure~\ref{fig:Q-module} in Section~\ref{sec:TypeA}.

Given $I,J\subseteq S$, define $\bN_{I,J}^S$ to be the $\FF$-span of $\pi_w\pib_{w_0(S\setminus I)}$ for all $w\in W\setminus W_J$ with $D(w)=I$, and define $\bQ_{I,J}^S:=\bP_I^S\Big/\bN_{I,J}^S$.
With $[a]$ denoting the image of $a\in\bP_I^S$ in $\bQ_{I,J}^S$, we have the following $\FF$-basis for $\bQ_{I,J}^S$:
\begin{equation}\label{eq:Q}
\left\{ [\pi_w\pib_{w_0(S\setminus I)}]: w\in W_J,\ D(w)=I \right\}.
\end{equation}

If there exists $w\in W_J$ with $D(w)=I$ then $I\subseteq J$. 
Thus $\bQ_{I,J}^S=0$ unless $I\subseteq J$.
Since the descent class of $I$ in $W$ is an interval between $w_0(I)$ and $w_1(I)$ under the left weak order, and the only element $w\in W_I$ with $D(w)=I$ is $w_0(I)$, we have $\bN_{I,I}^S=\rad(\bP_I^S)$ and $\bQ_{I,I}^S=\bC_I^S$.
The general result on $\bQ_{I,J}^S$ is below.

\begin{lemma}\label{lem:QIJ}
Assume $I\subseteq J\subseteq S$. 
Then $\bQ_{I,J}^S$ is an indecomposable $\H_S(0)$-module with $\top(\bQ_{I,J}^S) \cong \bC_I^S$, nonprojective unless $J=S$, and isomorphic to $\bP_I^J$ as an $\H_J(0)$-module with $\pi_s\bQ_{I,J}^S=0$ for all $s\in S\setminus J$.
\end{lemma}

\begin{proof}
By the equation~\eqref{eq:H0Action}, $\bN_{I,J}^S$ is an $\H_S(0)$-submodule of $\bP_I^S$.
Thus the quotient $\bQ_{I,J}^S$ of $\bP_I^S$ by this submodule is an $\H_S(0)$-module.
If $s\in S\setminus J$ then $\pi_s\bP_I^S\subseteq \bN_{I,J}^S$ by the equation~\eqref{eq:H0Action} and thus $\pi_s\bQ_{I,J}^S=0$.
Comparing the basis~\eqref{eq:Q} for $\bQ_{I,J}^S$ with the basis $\left\{ \pi_w\pib_{w_0(J\setminus I)}: w\in W_J,\ D(w)=I \right\}$ for $\bP_I^J$, we have a vector space isomorphism $\bQ_{I,J}^S\cong \bP_I^J$ which preserves $H_J(0)$-actions.
It follows from Proposition~\ref{prop:induce} (i) that $\bQ_{I,J}^S$ is an indecomposable $\H_S(0)$-module.
The top of $\bQ_{I,J}^S$ is isomorphic to $\bC_I^S$ since
\[ \rad\left(\bQ_{I,J}^S\right) = \rad\left(\H_S(0)\right) \left(\bP_I^S \Big/ \bN_{I,J}^S\right) = \rad\left(\bP_I^S\right) \Big/ \bN_{I,J}^S. \]
Thus if $\bQ_{I,J}^S$ is projective then it must be isomorphic to $\bP_I^S$, which forces $J=S$.
\end{proof}


Lastly, we recall from our earlier work~\cite[\S2.3]{H0Ch} the induction and restriction formulas for representations of 0-Hecke algebras.
Let $I\subseteq J\subseteq S$ and let $w$ be any element of $W$ with $D(w)=I$. 
The equalities
\begin{equation}\label{eq:H0Ind}
\bP_I^J \uparrow\,_{\H_J(0)}^{\H_S(0)}\, =  \sum_{K\subseteq S\setminus J} \bP_{I\cup K} \qand
\bC_I^J \uparrow\,_{\H_J(0)}^{\H_S(0)}\, =  \sum_{z\in\, ^JW} \bC_{D(wz)}^S. 
\end{equation}
hold in the Grothendieck groups $K_0(\H_S(0))$ and $G_0(\H_S(0))$, respectively.
If $K\subseteq S$ then the equalities
\begin{equation}\label{eq:H0Res}
\bP_K^S \downarrow\,_{\H_J(0)}^{\H_S(0)}\, =  \sum_{K'\,\in\, K\,\downarrow\,_J^S} \bP_{K'}^J \qand
\bC_K^S \downarrow\,_{\H_J(0)}^{\H_S(0)}\, =  \bC_{J\cap K}^J 
\end{equation}
hold in the Grothendieck groups $K_0(\H_J(0))$ and $G_0(\H_J(0))$, respectively, where $K\downarrow\,_J^S$ consists of certain subsets of $S$ that can be explicitly determined by a result from our earlier work~\cite[Prop.~17]{H0Ch}.
Furthermore, the following two-sided duality holds for induction and restriction of $0$-Hecke modules:
\begin{equation}\label{eq:H0Duality}
\left\langle \bP_I^J \uparrow\,_{\H_J(0)}^{\H_S(0)},\ \bC_K^S \right\rangle 
= \left\langle \bP_I^J,\ \bC_K^S\downarrow\,_{\H_J(0)}^{\H_S(0)} \right\rangle
\qand
\left\langle \bP_K^S \downarrow\,_{\H_J(0)}^{\H_S(0)},\ \bC_I^J \right\rangle 
= \left\langle \bP_K^S,\ \bC_I^J\uparrow\,_{\H_J(0)}^{\H_S(0)} \right\rangle. 
\end{equation}

\subsection{The symmetric groups and $0$-Hecke algebras of type $A$}\label{sec:TypeA}
In this subsection we summarize the representation theory of the type $A$ Coxeter groups (i.e., symmetric groups) and $0$-Hecke algebras, as well as the connections to combinatorics.
We put all these in a more general framework using the notion of \emph{Grothendieck groups of a tower of algebras} $A_*: A_0\hookrightarrow A_1\hookrightarrow A_2\hookrightarrow\cdots$, defined as
\[ G_0(A_*) := \bigoplus_{n\ge0} G_0(A_n) \qand
K_0(A_*) := \bigoplus_{n\ge0} K_0(A_n). \]
Let $M$ and $N$ be finitely generated (projective) modules over $A_m$ and $A_n$, respectively.
Extending the duality between $G_0(A_i)$ and $K_0(A_i)$ for each fixed $i$, we define $\langle M, N\rangle:=0$ if $m\ne n$.
Also define 
\begin{equation}\label{eq:prod}
M\htimes N := \left( M\otimes N \right) \uparrow\,_{A_m\otimes A_n}^{A_{m+n}} \qand \Delta(M) := \sum_{0\le i\le m} M \downarrow\,_{A_i\otimes A_{m-i}}^{A_m}.  
\end{equation}
Bergeron and Li~\cite{BergeronLi} showed that, if $A_*$ satisfies certain conditions, then with the pairing $\langle-,-\rangle$, the Grothendieck groups $G_0(A_*)$ and $K_0(A_*)$ become dual graded Hopf algebras whose product and coproduct are defined by \eqref{eq:prod}.

The symmetric group $\SS_n$ consists of all permutations on the set $[n]:=\{1,2,\ldots,n\}$.
It is generated by the \emph{adjacent transpositions} $s_1,\ldots,s_{n-1}$, where $s_i:=(i,i+1)$, with the quadratic relations $s_i^2=1$ for all $i\in[n-1]$ as well as the braid relations $s_is_{i+1}s_i=s_{i+1}s_is_{i+1}$ for all $i\in[n-2]$ and $s_is_j=s_js_i$ whenever $1\le i,j<n$ and $|i-j|>1$.
The group $W=\SS_n$ and the set $S=\{s_1,\ldots,s_{n-1}\}$ form the finite irreducible Coxeter system of type $A_{n-1}$.
The descent set of $w\in\SS_n$ is $D(w)=\{i\in[n-1]:w(i)>w(i+1)\}$ where we identify $s_i$ with $i$.
The length of $w\in\SS_n$ is $\ell(w)=\{(i,j): 1\le i<j\le n,\ w(i)>w(j)\}$.

A \emph{partition} $\lambda=[\lambda_1,\ldots,\lambda_\ell]$ is a weakly decreasing sequence of positive integers $\lambda_1\ge\cdots\ge\lambda_\ell$.
We use square brackets for partitions to distinguish them from compositions (defined later).
The \emph{size} of $\lambda$ is $|\lambda|:=\lambda_1+\cdots+\lambda_\ell$.
The \emph{length} of $\lambda$ is $\ell(\lambda):=\ell$.
We say $\lambda$ is a partition of $n=|\lambda|$ and write $\lambda\vdash n$.
The Grothendieck group $G_0(\CC\SS_*)=K_0(\CC\SS_*)$ of the tower of algebras $\CC\SS_*: \CC\SS_0\hookrightarrow \CC\SS_1\hookrightarrow\CC\SS_2\hookrightarrow\cdots$ is a free abelian group with a basis $\{\bS_\lambda:\lambda\vdash n,\ n\ge0\}$.
There exist integers $c^\lambda_{\mu,\nu} \ge0$, known as the \emph{Littlewood-Richardson coefficients}, for all $\lambda\models m+n$, $\mu\vdash m$, and $\nu\vdash n$ such that
\begin{equation*}
\left( \bS_\mu \otimes \bS_\nu \right)\uparrow\,_{\SS_m\otimes \SS_n}^{\SS_{m+n}} \cong \bigoplus_{\lambda\vdash m+n} \bS_\lambda^{\oplus c_{\mu,\nu}^\lambda } \qand
\bS_\lambda \downarrow\,_{\SS_m\otimes \SS_n}^{\SS_{m+n}} \cong \bigoplus_{ \substack{ \mu\vdash m \\ \nu\vdash n} } \left( \bS_\mu\otimes \bS_\nu \right)^{\oplus c_{\mu,\nu}^\lambda}.
\end{equation*}
It follows from the above formulas that, with the product $\widehat\otimes$ and coproduct $\Delta$ defined in~\eqref{eq:prod}, the Grothendieck group $G_0(\CC\SS_*)$ becomes a self-dual graded Hopf algebra, which is isomorphic to the Hopf algebra $\Sym$ of symmetric functions via the \emph{Frobenius characteristic map} defined by sending $\bS_\lambda$ to the Schur function $s_\lambda$ for all partitions $\lambda$.
The antipode is defined by $\bS_\lambda\mapsto (-1)^{|\lambda|} \bS_{\lambda^t}$ for all partitions $\lambda$, where $\lambda^t$ is the \emph{transpose} of $\lambda$.
 See, e.g., Grinberg and Reiner~\cite[\S4.4]{GrinbergReiner} for more details.

A \emph{composition} $\alpha=(\alpha_1,\ldots,\alpha_\ell)$ is a sequence of positive integers.
The \emph{size} of $\alpha$ is $|\alpha|:=\alpha_1+\cdots+\alpha_\ell$ and the \emph{length} of $\alpha$ is $\ell(\alpha):=\ell$.
The \emph{parts} of $\alpha$ are $\alpha_1,\ldots,\alpha_\ell$.
If $|\alpha|=n$ then we say $\alpha$ is a \emph{composition of $n$} and write $\alpha\models n$. 
Sending $\alpha$ to its \emph{descent set} 
$D(\alpha):=\{\alpha_1,\alpha_1+\alpha_2,\ldots,\alpha_1+\cdots+\alpha_{k-1}\}$
gives a bijection between compositions of $n$ and subsets of $[n-1]$.
If $\alpha=(\alpha_1,\ldots,\alpha_\ell)\models m$ and $\beta=(\beta_1,\ldots,\beta_k)\models n$ then we have compositions 
$\alpha\cdot\beta := (\alpha_1,\ldots,\alpha_\ell,\beta_1,\ldots,\beta_k)$ and 
$\alpha\rhd\beta := (\alpha_1,\ldots,\alpha_{\ell-1},\alpha_\ell+\beta_1,\beta_2,\ldots,\beta_k)$ of $m+n$.

Given $I,J\subseteq S$, there exist unique compositions $\alpha$ and $\beta$ of $n$ such that $D(\alpha)=I$ and $D(\beta)=S\setminus J$. 
Let $\bP_\alpha:=\bP_I^S$, $\bC_\alpha:=\bC_I^S$, $\bN_{\alpha,\beta}:=\bN_{I,J}^S$, and $\bQ_{\alpha,\beta}:=\bQ_{I,J}^S$ (see \S~\ref{sec:H0}).
Examples are given in Figure~\ref{fig:Q-module}.
 
\begin{figure}[h]
\begin{center}
$\xymatrix { 
 \bP_{(1,3)}: & *+{\pi_1\pib_2\pib_3\pib_2} \ar@(ld,rd)[]_{\pi_1=1,\pi_3=0} \ar[r]^-{\pi_2} & 
 *+{\pi_2\pi_1\pib_2\pib_3\pib_2} \ar@(ld,rd)[]_{\pi_1=0,\pi_2=1} \ar[r]^-{\pi_3} &
 *+{\pi_3\pi_2\pi_1\pib_2\pib_3\pib_2} \ar@(ld,rd)[]_{\pi_1=\pi_2=0,\pi_3=1}
 }$

\vspace{3pt}

$\bC_{(1,3)}$ and $\bN_{(1,3),(2,1,1)}$:
$\xymatrix { 
& *+{[\pi_1\pib_2\pib_3\pib_2]} \ar@(ld,rd)[]_{\pi_1=1,\pi_2=\pi_3=0} &
*+{\pi_2\pi_1\pib_2\pib_3\pib_2} \ar@(ld,rd)[]_{\pi_1=0,\pi_2=1} \ar[r]^-{\pi_3} &
 *+{\pi_3\pi_2\pi_1\pib_2\pib_3\pib_2} \ar@(ld,rd)[]_{\pi_1=\pi_2=0,\pi_3=1}
}$

\vspace{3pt}

$\bQ_{(1,3),(3,1)}$ and $\bN_{(1,3),(3,1)}$:
$\xymatrix { 
& *+{[\pi_1\pib_2\pib_3\pib_2]} \ar@(ld,rd)[]_{\pi_1=1,\pi_3=0} \ar[r]^-{\pi_2} & 
 *+{[\pi_2\pi_1\pib_2\pib_3\pib_2]} \ar@(ld,rd)[]_{\pi_1=\pi_3=0,\pi_2=1} & 
*+{\pi_3\pi_2\pi_1\pib_2\pib_3\pib_2} \ar@(ld,rd)[]_{\pi_1=\pi_2=0,\pi_3=1} 
}$
\end{center}
\caption{Some examples of $\H_4(0)$-modules}
\label{fig:Q-module}
\end{figure}

The Grothendieck groups $G_0(\H_*(0))$ and $K_0(\H_*(0))$ of the tower $\H_*(0):\H_0(0)\hookrightarrow\H_1(0)\hookrightarrow\H_2(0)\hookrightarrow\cdots$ of algebras are free abelian groups with bases $\{\bC_\alpha: \alpha\models n,\ n\ge0\}$ and $\{\bP_\alpha: \alpha\models n,\ n\ge0\}$, respectively, where $\H_0(0):=\FF$.
With the product $\widehat\otimes$ and the coproduct $\Delta$ given by~\eqref{eq:prod}, $G_0(\H_*(0))$ and $K_0(\H_*(0))$ become graded Hopf algebras, which are dual to each other by the two-sided duality \eqref{eq:H0Duality}.
By Krob and Thibon~\cite{KrobThibon}, there is an isomorphism between $G_0(\H_*(0))$ [or $K_0(\H_*(0))$, resp.] and the Hopf algebra $\QSym$ of \emph{quasisymmetric functions} [or the Hopf algebra $\NSym$ of \emph{noncommutative symmetric functions}, resp.].
The antipode maps are defined by $\bC_\alpha\mapsto (-1)^{|\alpha|} \bC_{\alpha^t}$ and $\bP_\alpha\mapsto (-1)^{|\alpha|} \bP_{\alpha^t}$, respectively, for all compositions $\alpha$, where $\alpha^t$ is the \emph{transpose} of $\alpha$.
See, e.g., Grinberg and Reiner~\cite{GrinbergReiner} for details.

We next recall the formulas for $\widehat\otimes$ and $\Delta$.
Let $\alpha\models m$ and $\beta\models n$. 
For any $u\in\SS_m$ and $v\in\SS_n$ with $D(u)=D(\alpha)$ and $D(v)=D(\beta)$, let $u\shuffle v$ be the set of all permutations in $\SS_{m+n}$ obtained by shuffling $u(1),\ldots,u(m)$ and $v(1)+m,\ldots,v(n)+m$; this is the \emph{(shifted) shuffle product} of permutations.
Let $\alpha\shuffle\beta$ be the multiset of compositions of $m+n$ in bijection with $u\shuffle v$ via the descent map;
this definition does not depend on the choice of $u$ and $v$.
For example, the elements of the multiset $(2)\shuffle(2)$ are $(4), (2,2), (3,1), (1,3), (1,2,1), (2,2)$ since $12\shuffle 12 = \{1234, 1324, 1342, 3124, 3142, 3412\}$.
One has
\begin{equation*}
\left( \bP_\alpha\otimes\bP_\beta \right) \uparrow\,_{\H_m(0)\otimes\H_n(0)}^{\H_{m+n}(0)} 
\cong \bP_{\alpha\beta} \oplus \bP_{\alpha\rhd\beta}
\qand 
\left( \bC_\alpha\otimes\bC_\beta \right)\uparrow\,_{\H_m(0)\otimes\H_n(0)}^{\H_{m+n}(0)} 
= \bigoplus_{\gamma\in \alpha\shuffle \beta} \bC_\gamma
\end{equation*}
where $\bP_{\alpha\rhd\beta}$ is treated as $0$ if $\alpha$ or $\beta$ is the empty composition.
On the other hand, if $\alpha\models m+n$ then
\[  \bP_\alpha\downarrow\,_{\H_m(0)\otimes\H_n(0)}^{\H_{m+n}(0)} 
\cong \bigoplus_{(\beta,\gamma) \in\, \alpha\downarrow_m} \bP_\beta\otimes\bP_\gamma
\qand \bC_\alpha\downarrow\,_{\H_m(0)\otimes\H_n(0)}^{\H_{m+n}(0)} 
\cong \bC_{\alpha_{\le m}}\otimes\bC_{\alpha_{>m}} \]
where $\alpha\downarrow_m$ is a multiset  consisting of certain pairs $(\beta,\gamma)$ of compositions $\beta\models m$ and $\gamma\models n$ constructed in our earlier work~\cite[Proposition~4.5]{H0Tab} and $\alpha_{\le m}$ and $\alpha_{>m}$ are the unique compositions of $m$ and $n$, respectively, such that $\alpha\in \{\alpha_{\le m} \alpha_{>m}, \alpha_{\le m}\rhd\alpha_{>m}\}$ (e.g., $\alpha=(2,1,3,1)$, $\alpha_{\le 4}=(2,1,1)$, $\alpha_{>4}=(2,1)$).

\section{Structure and dimension}\label{sec:dim}

Let $(W,S)$ be a Coxeter system and let $\FF$ be an arbitrary field.
The \emph{Hecke algebra} $\H_S(\bq)$ of $(W,S)$ with \emph{independent parameters} $\bq:=(q_s:s\in S)\in\FF^S$ is an $\FF$-algebra generated by $\{T_s:s\in S\}$ with 
\begin{itemize}
\item
quadratic relations $(T_s-1)(T_s+q_s)=0$ for all $s\in S$, and 
\item
braid relations $(T_sT_tT_s\cdots)_{m_{st}}=(T_tT_sT_t\cdots)_{m_{st}}$ for all distinct $s,t\in S$.
\end{itemize}
Taking $q_s=q\in\FF$ for all $s\in S$ in the definition of $\H_S(\bq)$ gives the usual \emph{Hecke algebra} $\H_S(q)$ of $(W,S)$ over $\FF$ with a single parameter $q$.
When $\FF=\CC$ and $q\in\CC\setminus\{0, \text{roots of unity} \}$, there exists an algebra isomorphism $\phi: \H_S(q) \cong \CC W$ by a general deformation argument of Tits or by an explicit construction of Lusztig~\cite{Lusztig81}.
If one only insists $q_s=q_t$ whenever $m_{st}$ is odd, then $\H_S(\bq)$ becomes a \emph{Hecke algebra with unequal parameters} studied by Lusztig~\cite{Lusztig03}.

\subsection{Previous results}
In this subsection we summarize the main results of our earlier work~\cite{Hecke} on the Hecke algebra $\H_S(\bq)$ with $\bq\in\FF^S$ arbitrary.
Let $w\in W$ with a reduced expression $w=s_1\cdots s_k$. 
Then $T_w:= T_{s_1} \cdots T_{s_k}$ is well defined, thanks to the Word Property of $W$~\cite[Theorem~3.3.1]{BjornerBrenti}.
If $s\in S$ then
\begin{equation}\label{eq:Tsw}
T_sT_w = \begin{cases}
T_{sw}, & \text{if } \ell(sw)>\ell(w),\\
q_s T_{sw} + (1-q_s)T_w, & \text{if } \ell(sw)<\ell(w).
\end{cases} 
\end{equation}
The set $\{T_w:w\in W\}$ always spans $\H_S(\bq)$.
This spanning set is indeed a basis if and only if $\H_S(\bq)$ is a Hecke algebra with unequal parameters, i.e., $q_s=q_t$ whenever $m_{st}$ is odd~\cite[Theorem~1.2]{Hecke}.

For any subset $R\subseteq S$, we use $\H_R(\bq) = \H_R(\bq|_R)$ to denote the Hecke algebra of the Coxeter system $(W_R,R)$ with independent parameters $(q_r:r\in R)$.
We warn the reader that $\H_R(\bq)$ is not necessarily isomorphic to the subalgebra of $\H_S(\bq)$ is generated by $\{T_r:r\in R\}$~\cite[\S3]{Hecke}.

The \emph{collapsed subset} $R\subseteq S$ consists of all $s\in S$ connected to some other $t\in S$ with $q_s\ne q_t$ via a path in the Coxeter diagram of $(W,S)$ whose edges all have odd weights and whose vertices (including $s$ and $t$) all correspond to nonzero parameters.
We have~\cite[Theorem~3.2]{Hecke}
\begin{center}
(1) $T_r=1$ for all $r\in R$,\quad (2) $T_s\notin \FF$ for all $s\in S\setminus R$,\quad and \quad (3) $\H_S(\bq) \cong \H_{S\setminus R}(\bq)$.
\end{center}
Thus we may assume, without loss of generality, that $\H_S(\bq)$ is \emph{collapse free}, meaning that $q_sq_t=0$ whenever $q_s\ne q_t$ and $m_{st}$ is odd. 
We will keep this assumption throughout the paper.


\begin{lemma}\cite{Hecke}\label{lem:01=0}
Suppose there exists a path $(s=s_0,s_1,s_2,\ldots,s_k=t)$ consisting of simply-laced edges in the Coxeter diagram of $(W,S)$, where $k\geq1$. If $q_s=0$ and $q_{s_i}\ne0$ and $m_{ss_i}\leq 3$ for all $i\in[k]$, then $T_sT_t=T_tT_s=T_s$.
\end{lemma}

Lemma~\ref{lem:01=0} played an important role in our derivation of the following results~\cite{Hecke}.
First, the algebra $\H_S(\bq)$ is commutative if and only if the Coxeter diagram of $(W,S)$ is simply laced and exactly one of $q_s$ and $q_t$ is zero whenever $m_{st}=3$.
Next, a commutative $\H_S(\bq)$ has a basis indexed by the independent sets in the Coxeter diagram of $(W,S)$, which is a simple bipartite graph in this case. 
In particular, the dimension of $\H_S(\bq)$ is the \emph{Fibonacci number} $F_{n+2}:=F_{n+1}+F_n$ with $F_0:=0$ and $F_1:=1$ when $(W,S)$ is of type $A_n$ for all $n\ge1$, or the \emph{Lucas number} $L_n:=F_{n+1}+F_{n-1}$ when $(W,S)$ is of affine type $\widetilde A_n$ for all even $n\ge4$.
We conjectured that if the Coxeter diagram of $(W,S)$ is a simple bipartite graph then the minimum dimension of $\H_S(\bq)$ is attained when $\H_S(\bq)$ is commutative and verified this conjecture for type $A$.
We also constructed a basis for $\H_S(\bq)$ in the special case when $(W,S)$ is simply laced.

\begin{theorem}\cite{Hecke}\label{thm:basis}
Suppose $(W,S)$ is simply laced and $\H_S(\bq)$ is collapse free.
Then the following statements hold.
\begin{enumerate}
\item
The set $S$ decomposes into a disjoint union of subsets $S_1,\ldots,S_k$ such that the elements of each $S_i$ receive the same parameter and are connected in the Coxeter diagram of $(W,S)$, and that if $s\in S_i$, $t\in S_j$, $i\ne j$, then either $m_{st}=2$ or exactly one of $q_s$ and $q_t$ is zero.
\item
There is a basis for $\H_S(\bq)$ consisting of all elements of the form $T_{w_1}\cdots T_{w_k}$, where $w_i\in W_{S_i}$ for $i=1,\ldots,k$ and if there exist $s\in S_i$ and $t\in S_j$ with $i\ne j$ such that $q_s=0$, $m_{st}=3$, and $s$ occurs in some reduced expression of $w_i$, then $w_j=1$.
\end{enumerate}
\end{theorem}

\begin{example}\label{ex:MixedType}
For $\H_S(\bq)$ represented below, where $c\in\FF\setminus\{0\}$, we have a partition $S=S_1\sqcup S_2\sqcup S_3$ with $S_1$ of type $D_4$, $S_2$ of type $E_7$, and $S_3$ of type $A_2$.
We will compute the dimension of $\H_S(\bq)$ later.
\[ \xymatrix @R=0pt @C=10pt {
1 \ar@{-}[rd] & & & & & 0 \ar@{-}[r] & 0 \ar@{-}[r] & c \ar@{-}[r] & c\\
& 1 \ar@{-}[r] & 1 \ar@{-}[r] & 0 \ar@{-}[r] & 0 \ar@{-}[rd] \ar@{-}[ru] \\ 
1 \ar@{-}[ru] & & & & & 0 \ar@{-}[r] & 0 \ar@{-}[r] & 0 
} \]
 \end{example} 

\begin{example}
Let $(W,S)$ be the Coxeter system of type $A_n$, i.e., $W=\SS_{n+1}$ and $S=\{s_1,\ldots,s_n\}$.
We can view $\bq\in\FF^S$ as a vector $(q_1,\ldots,q_n)\in\FF^n$ whose $i$th component is the parameter for $s_i$.
Thus we can write $\H(q_1,\ldots,q_n):=\H_S(\bq)$.
For instance, the Hecke algebra $\H(0,0,1)$ of the Coxeter system $(W,S)$ of type $A_3$ with independent parameters $(q_1,q_2,q_3)=(0,0,1)$ is generated by $T_1,T_2,T_3$ and has dimension $6+2=8$ since by Theorem~\ref{thm:basis} it has a basis $\{T_wT_u\}$, where $w\in\SS_3$ and $u\in\SS_2$ satisfy the condition that if $s_2$ occurs in some reduced expression of $w$ then $u=1$.
\end{example}

\subsection{New results in the simply-laced case} 
In this paper we focus on the Hecke algebra $\H_S(\bq)$ of a finite simply-laced Coxeter system $(W,S)$ with independent parameters $\bq\in\FF^S$.
We may assume $\H_S(\bq)$ is collapse free.
We further assume that $q_1,\ldots,q_\ell$ are not roots of unity to avoid technicalities.
It would still be interesting to explore the case when $q_1,\ldots,q_\ell$ are allowed to be roots of unity in the future.

\begin{definition}
Let $S=S_1\sqcup \cdots \sqcup S_k$ be the partition given by Theorem~\ref{thm:basis}.
For each $i\in [k]$, we write $W_i := \langle S_i \rangle$.
There exists a partition $[k] = L_0 \sqcup L_1$ such that $q_s = 0$ for all $s\in S_i$, $i\in L_0$, and that $q_t \ne 0$ (we can actually assume $q_t=1$ by Proposition~\ref{prop:q->1} below) for all $t\in S_j$, $j\in L_1$.
Define $W^0 := \langle S^0 \rangle$ and $W^1 := \langle S^1 \rangle$, where
\[ S^0 := \{ s\in S: q_s=0\} = \bigsqcup_{i\in L_0} S_i \qand
S^1 := \{ t\in S: q_t\ne0\} = \bigsqcup_{j\in L_1} S_j. \]
Given $J\subseteq L_1$ and $i\in L_0$, define $\overline W_i^J$ to be the parabolic subgroup of $W_i$ generated by 
\[ \overline S_i^J := \{s\in S_i: m_{st}=2 \text{ whenever } t\in S_j,\ j\in J\}. \]
\end{definition}

By Lemma~\ref{lem:01=0} and Theorem~\ref{thm:basis}, we have the following alternative description for $\H_S(\bq)$. 
\begin{enumerate}
\item
The subalgebra $\H^0(\bq)$ of $\H_S(\bq)$ generated by $\{T_s: s\in S^0\}$ is isomorphic to $\H_{S^0}(0) \cong \bigotimes_{i\in L_0} \H_{S_i}(0)$.
\item
The subalgebra $\H^1(\bq)$ of $\H_S(\bq)$ generated by $\{T_t: t\in S^1\}$ is isomorphic to $\CC W^1 \cong \bigotimes_{j\in L_1} \CC W_j$.
\item
The two subalgebras $\H^0(\bq)$ and $\H^1(\bq)$ commute.
\item
If $s\in S^0$ and $t\in S^1$ satisfy $m_{st}=3$ then $T_sT_t = T_tT_s = T_s$.
\end{enumerate}
It follows that
\begin{equation}\label{eq:Hq}
\H_S(\bq) =  \H^0(\bq)\H^1(\bq) \cong \H_{S^0}(0)  \otimes \CC W^1 \Big/ \left( T_sT_t-T_s: s\in S^0,\ t\in S^1,\ m_{st}=3 \right).
\end{equation}

\begin{proposition}\label{prop:dim}
The dimension of $\H_S(\bq)$ equals
\[ \sum_{J\subseteq L_1} \prod_{j\in J}(|W_j|-1) \prod_{i\in I} \left| \overline W_i^J \right|. \]
\end{proposition}

\begin{proof}
Let $W_S(\bq)$ denote the set of all tuples $(w\in W_i: i\in [k])$ such that $w_i\in \overline W_i^J$ for each $i\in L_0$, where $J := \{ j\in L_1: w_j\ne 1 \}$.
By Theorem~\ref{thm:basis}, $\H_S(\bq)$ has a basis $\left\{ T_{w_1}\cdots T_{w_k} : (w_1,\cdots,w_k) \in W_S(\bq) \right\}$.
For each $k$-tuple $(w_1,\ldots,w_k) \in W_S(\bq)$, we define $\phi(w_1,\ldots,w_k):=\{j\in L_1: w_j\ne 1\}$.
Summing up the cardinalities of the fibers of all subsets of $L_1$ under the map $\phi$ gives the dimension of $\H_S(\bq)$.
\end{proof}

\begin{example}
We revisit the algebra $\H_S(\bq)$ in Example~\ref{ex:MixedType}.
We have $L_0=\{2\}$ and $L_1=\{1,3\}$. 
By Proposition~\ref{prop:dim} and the tables below, the dimension of $\H_S(\bq)$ is
\[ 1\cdot 72\cdot 8! + 191\cdot 7! + 5\cdot 2^5\cdot 6! + 191\cdot5\cdot6! = 4\cdot1621\cdot 6! = 4668480. \]
\begin{center}
\begin{tabular}{|c|c|c|}
\hline
Group & Type & Order \\
\hline
$W_1$ & $D_4$ & $192$ \\
\hline
$W_2$ & $E_7$ & $72\cdot 8!$ \\
\hline
$W_3$ & $A_2$ & $6$ \\
\hline
\end{tabular}
\qquad
\begin{tabular}{|c|c|c|}
\hline
$J$ & $\prod_{j\in J} (|W_j|-1)$  & $\left|\overline{W}_2^J\right|$ \\
\hline
$\emptyset$ & $1$ & $72\cdot 8!$ \\
\hline
$\{1\}$ & $191$ & $7!$ ($A_6$) \\
\hline
$\{3\}$ & $5$ & $2^5\cdot6!$ ($D_6$) \\
\hline
$\{1,3\}$ & $191\cdot 5$ & $6!$ ($A_5$) \\
\hline
\end{tabular}
\end{center} 
\end{example}

\begin{example}
By Proposition~\ref{prop:dim}, for any positive integers $a,b,c\ge1$, we have
\[ \dim \H(0^a1^b0^c) = (a+1)!(c+1)! + a!((b+1)!-1)c! \qand \]
\[ \dim \H(1^a0^b1^c) = (b+1)!+((a+1)!-1)b!+b!((c+1)!-1)+((a+1)!-1)(b-1)!((c+1)!-1).\]
In earlier work~\cite{Hecke} we gave these two formulas and also showed that, for $n\ge0$, if $\bq$ is an alternating sequence in $\{0,1\}$ of length $n$, then $\H(\bq)$ is a commutative algebra whose dimension equals the \emph{Fibonacci number} $F_{n+2} := F_{n+1}+F_n$ with initial terms $F_0:=0$ and $F_1:=1$. 
Now combining this with Proposition~\ref{prop:dim} we have, for any integers $k,r\ge0$ and $n\ge1$,
\[ \dim \H(\underbrace{\cdots1010}_k 1^{n-1} \underbrace{0101\cdots}_{r}) = F_{k+2}F_{r+2} + (n!-1)F_{k+1}F_{r+1}.\]
This recovers a well-known identity $F_{k+2}F_{r+2}+F_{k+1}F_{r+1} = F_{k+r+3}$ when $n=2$, and gives the number $F_{k+2}+(n!-1)F_{k+1} = F_k + n! F_{k+1}$ when $r=0$, which satisfies the usual Fibonacci recurrence with initial terms $1$ and $n!$ (see OEIS~\cite[A022096 and A022394]{OEIS} for $n=3,4$).
We also have
\begin{align*}
\dim \H(\underbrace{\cdots0101}_k 0^{m-1} \underbrace{1010\cdots}_{r}) 
&= F_{k+1}(m!+2(m-1)!+(m-2)!)F_{r+1} \\
&= (m^2+m-1)(m-2)!F_{k+1}F_{r+1}.
\end{align*}
\end{example}

Next, using the algebra isomorphism $\phi: \CC W^1 \cong \H^1(\bq)$ given by either Tits or Lusztig~\cite{Lusztig81} together with the algebra homomorphism $c: \H^1(\bq) \to \CC$ defined by $c(T_t)=1$ for all $t\in S^1$, we show that each parameter $q_s\in \CC\setminus\{0,\ \text{roots of unity}\}$ of the algebra $\H_S(\bq)$ can be assumed to be $1$, without loss of generality.

\begin{proposition}\label{prop:q->1}
Let $\H_S(\bq)$ be the Hecke algebra over $\FF=\CC$ of a finite simply-laced Coxeter system $(W,S)$ with independent parameters $\bq:= \left( q_s\in \CC\setminus\{ \text{roots of unity} \}: s\in S \right)$.
Then $\H_S(\bq)$ is isomorphic to the algebra $\H_S(\bq')$, where $\bq'=(q'_s:s\in S)$ is defined by
\[ q'_s := 
\begin{cases}
0, & \text{if } q_s = 0, \\
1, & \text{if } q_s \ne 0.
\end{cases} \]
\end{proposition}

\begin{proof}
Let $\{T_s:s\in S\}$ and $\{T'_s: s\in S\}$ be the generating sets of $\H_S(\bq)$ and $\H_S(\bq')$ given by the definition of the two algebras.
For each $s\in S^0$ define $T''_s := T_s$.
For each $t\in S^1$ define $T''_t:=c_t\phi(t)$, where $c_t:=c(\phi(t))=\pm1$ since $\phi(t)^2=1$.\footnote{
Lusztig~\cite{Lusztig81} gives an explicit isomorphism between $\H_S(q)$ and $\CC W$;
it is likely that the coefficient $c_t\in\{\pm1\}$ appearing in our proof can be determined using that isomorphism.}
One sees that $\{T''_s: s\in S\}$ is another generating set of $\H_S(\bq)$.
Since $\H_S(\bq)$ has the same dimension as $\H_S(\bq')$ by Proposition~\ref{prop:dim}, it suffices to show that the relations for $\{T'_s: s\in S\}$ are satisfied by $\{T''_s : s\in S\}$ as well.

It is clear that $\{T''_s: s\in S^0\} = \{T_s: s\in S^0\}$ satisfies the same relations as $\{T'_s: s\in S^0\}$.
Moreover, $\{T''_t: t\in S^1\}$ satisfies the relations for $\{T'_t: t\in S^1\}$ by the following argument.
\begin{itemize}
\item
For each $t\in S^1$, the relation satisfied by $T'_t$ is $(T'_t)^2=1$, and we also have $(T''_t)^2=c_t^2\phi(t)^2=1$.
\item
For any $r, t\in S^1$ with $m_{rt}=2$, the relation between $T'_s$ and $T'_t$ is the commutativity, which is also satisfied by $T''_r=c_r\phi(r)$ and $T''_t=c_t\phi(t)$ since $m_{rt}=2$ implies that $\phi(r)$ and $\phi(t)$ commute.
\item
For any $r, t\in S^1$ with $m_{rt}=3$ we have $\phi(r) \phi(t) \phi(r) = \phi(t) \phi(r) \phi(t)$ which implies $c_r c_t c_r = c_t c_r c_t$, and thus the braid relation between $T'_r$ and $T'_t$ is also satisfied by $T''_r=c_r\phi(r)$ and $T''_t=c_t\phi(t)$.
\end{itemize}
Finally, let $s\in S_i$ with $q_s=0$ and $t\in S_j$ with $q_t\ne0$.
Then $T''_t=c_t\phi(t)$ lies in the subalgebra of $\H_S(\bq)$ generated by $\{T_r: r\in S_j\}$ since $S_j$ is a connected component of the Coxeter diagram of $(W^1,S^1)$ by Theorem~\ref{thm:basis}.
If $m_{sr}=2$ for all $r\in S_j$ then the relation between $T'_s$ and $T'_t$ is the commutativity, which is also satisfied by $T''_s=T_s$ and $T''_t=c_t\phi(t)$ since $T_sT_r=T_rT_s$ for all $r\in S_j$.
Otherwise by Lemma~\ref{lem:01=0}, the relation between $T'_s$ and $T'_t$ is $T'_sT'_t=T'_s=T'_tT'_s$ and we also have
\[ T''_sT''_t = c_t T_s \phi(t) = c_t^2 T_s = T''_s = c_t^2 T_s = c_t \phi(t) T_s = T''_t T''_s \]
where $T_s\phi(t)=c_t T_s=\phi(t)T_s$ holds since $T_sT_r=T_s=T_rT_s$ for all $r\in S_j$.
\end{proof}

\section{Simple and projective indecomposable modules}\label{sec:decomp}

Let $\H_S(\bq)$ be the Hecke algebra of a finite simply-laced Coxeter system $(W,S)$ over the complex field $\FF=\CC$ with independent parameters $\bq\in(\CC\setminus\{\text{roots of unity}\})^S$.
In this section we construct all simple $\H_S(\bq)$-modules and projective indecomposable $\H_S(\bq)$-modules, and use them to determine the quiver and representation type of $\H_S(\bq)$.

By Proposition~\ref{prop:q->1}, we may assume $\bq\in\{0,1\}^S$, without loss of generality. 
Recall that $S$ can be partitioned into $S=S_1\sqcup\cdots\sqcup S_k$ such that the elements of each $S_i$ are connected in the Coxeter diagram and all receive the same parameter.
There is also a partition $[k]=L_0\sqcup L_1$ such that $q_s = 0$ for all $s\in S_i$, $i\in L_0$, and that $q_t = 1$ for all $t\in S_j$, $j\in L_1$.

\subsection{Decomposition of the regular representation}
In this subsection we give a decomposition of the regular representation of $\H_S(\bq)$ and obtain all simple and projective indecomposable $\H_S(\bq)$-modules.

\begin{definition}
Let $\lambda\in \Irr(\CC W^1)$. 
We can write $\bS_\lambda = \bigotimes_{j\in L_1} \bS_{\lambda^j}$ where $\lambda^j\in\Irr(\CC W_j)$ for all $j\in L_1$.
We define $L_1^\lambda$ to be the set of all $j\in L_1$ such that $W_j$ acts on $\bS_\lambda$ nontrivially.
Then $\bS_\lambda = \bS_\lambda^t \otimes \bS_\lambda^n$ where 
\begin{equation}\label{eq:Stn}
\bS_\lambda^t := \bigotimes_{j\in L_1\setminus L_1^\lambda} \bS_{\lambda^j}
\qand \bS_\lambda^n := \bigotimes_{j\in L_1^\lambda} \bS_{\lambda^j}. 
\end{equation}
Let $I\subseteq S^0$ and let $S^{0,\lambda}$ denote the set of all $s\in S^0$ such that $m_{st}=2$ whenever $t\in S_j$, $j\in L_1^\lambda$.
Define $\bP_{I,\lambda}^S := \bP_I^{S^0} \bS_\lambda \subseteq \H_S(\bq)$, where $\bP_I^{S_0}$ is identified with a submodule of $\H_S^0(\bq)\cong\H_{S^0}(0)$ and $\bS_\lambda$ is identified with a submodule of $\H_S^1(\bq)\cong\CC W^1$.
\end{definition}

\begin{proposition}
Suppose $\lambda\in\Irr(\CC W^1)$ and $\bM$ is an $\H_{S^{0,\lambda}}(0)$-module.
Then $\bM\otimes \bS_\lambda$ becomes an $\H_S(\bq)$-module if we let $T_s$ act by zero for all $s\in S^0\setminus S^{0,\lambda}$, by its action on $\bM$ for all $s\in S^{0,\lambda}$, and by its action on $\bS_\lambda$ for all $s\in S^1$.
\end{proposition}

\begin{proof}
One can verify the defining relations of $\H_S(\bq)$ for the above $\H_S(\bq)$-action on $\bM\otimes \bS_\lambda$.
\end{proof}

\begin{lemma}\label{lem:PS}
Let $I\subseteq S^0$ and $\lambda\in\Irr(\CC W^1)$.
Identify $\bS_\lambda$ with a submodule of $\CC W^1\cong \H^1(\bq)\subseteq \H_S(\bq)$.
\begin{enumerate}
\item[(i)]
If $s\in S^0\setminus S^{0,\lambda}$ then $T_s\bS_{\lambda} = 0$.
Consequently, if $I\not\subseteq S^{0,\lambda}$ then $\bP_{I,\lambda}^S=0$.
\item[(ii)]
If $I\subseteq S^{0,\lambda}$ then we have an isomorphism  $\bP_{I,\lambda}^S \cong \bP_I^{S^{0,\lambda}} \otimes \bS_\lambda$ of $\H_S(\bq)$-modules.
\end{enumerate}
\end{lemma}

\begin{proof}
(i) If $s\in S^0\setminus S^{0,\lambda}$ then $m_{st}=3$ for some $t\in S_j$ with $j\in L_1^\lambda$, and it follows from Lemma~\ref{lem:01=0} that $T_s\bS_\lambda=0$ since $\bS_{\lambda^j}\subseteq \sigma(W_j)^\perp$ by the equation~\eqref{eq:DecompCG}.
If $I\not\subseteq S^{0,\lambda}$ then $\pi_{w_0(I)} = \pi_{w_0(I)s} \pi_s$ for some $s\in S^0\setminus S^{0,\lambda}$, and using $\pi_s\bS_\lambda = T_s\bS_\lambda=0$ we obtain
\[ \bP_{I,\lambda}^S = \H_{S^0}(0) \pi_{w_0(I)}\pib_{w_0(S^0\setminus I)} \bS_\lambda = 0.\]

\noindent (ii) Now assume $I\subseteq S^{0,\lambda}$.
Since $\pi_s$ and $\pib_s$ act on $\bS_\lambda$ by $0$ and $-1$, respectively, for all $s\in S^0\setminus S^{0,\lambda}$, one can use the multiplication rule~\eqref{eq:H0Prod} of the $0$-Hecke algebra to obtain
\[ \bP_{I,\lambda}^S = \H_{S^{0,\lambda}}(0) \pi_{w_0(I)}\pib_{w_0(S^{0,\lambda}\setminus I)} \bS_\lambda = \bP_I^{S^{0,\lambda}} \bS_\lambda.\]
We have $\bS_\lambda = \bS_\lambda^t \otimes \bS_\lambda^n$
where $\bS_\lambda^t$ is spanned by $\sigma:=\prod_{j\in L_1\setminus L_1^\lambda} \sigma(W_j)$.
Thus $\bP_I^{S^{0,\lambda}} \bS_\lambda^t$ is spanned by 
\[ \left\{ \pi_w\pib_{w_0(S^{0,\lambda}\setminus I)}\sigma: w\in \big\langle S^{0,\lambda} \big\rangle,\ D(w)=I \right\}. \]
This spanning set is indeed a basis, since the expansion of $\pi_w\pib_{w_0(S^{0,\lambda}\setminus I)}\sigma$ in terms of the basis of $\H_S(\bq)$ in Theorem~\ref{thm:basis} has a scalar multiple of $\pi_w$ as the leading term (i.e., the term with the smallest length) by Equation~\eqref{eq:pib} and Lemma~\ref{lem:01=0}.
Therefore $\bP_I^{S^{0,\lambda}} \bS_\lambda^t \cong \bP_I^{S^{0,\lambda}} \otimes \bS_\lambda^t$.
Combining this with the definition of $S^{0,\lambda}$, we have the desired isomorphism $\bP_{I,\lambda}^S \cong \bP_I^{S^{0,\lambda}} \otimes \bS_\lambda$.
If $s\in S^0\setminus S^{0,\lambda}$ then $T_s=\pi_s$ annihilates the left hand side of this isomorphism and hence the right hand side as well.
\end{proof}

\begin{theorem}\label{thm:decomp}
With $\Irr(\H_S(\bq)) := \left\{ (I,\lambda): \lambda\in \Irr(\CC W^1),\ I\subseteq S^{0,\lambda} \right\}$, we have a direct sum decomposition
\[ \H_S(\bq) \cong \bigoplus_{ (I,\lambda)\in\Irr(\H_S(\bq))} \left( \bP_{I,\lambda}^S \right)^{\oplus d_\lambda } \]
where each summand $\bP_{I,\lambda}^S$ is a projective indecomposable $\H_S(\bq)$-module satisfying
\[ \bP_{I,\lambda}^S \cong \bP_I^{S^{0,\lambda}} \otimes \bS_\lambda \qand 
\bC_{I,\lambda}^S := \top\bP_{I,\lambda} \cong \bC_I^{S^{0,\lambda}} \otimes \bS_\lambda. \]
\end{theorem}

\begin{proof}
We can write $\H_S(\bq) = \H^0(\bq)\H^1(\bq)$ as a sum of $d_\lambda$ copies of $\bP_{I,\lambda}^S$ for all $I\subseteq S^0$ and all $\lambda\in \Irr(\CC W^1)$ by applying the decompositions~\eqref{eq:DecompCG} and \eqref{eq:H0Decomp} to $\H_S^1(\bq)$ and $\H_S^0(\bq)$, respectively.
A summand $\bP_{I,\lambda}$ is nonzero if and only if $(I,\lambda)\in \Irr(\H_S(\bq))$ by Lemma~\ref{lem:PS}.
To show this is indeed a direct sum, we compute the dimension. 
For each $J\subseteq L_1$, the sum of the dimensions of the summands $\bP_{I,\lambda}^S$ satisfying $(I,\lambda)\in\Irr(\H_S(\bq))$ and $L_1^\lambda = J$ is
\[\prod_{j\in J} (|W_j|-1) \prod_{i\in L_0} \left| \overline W_i^J \right| .\]
Summing this up over all subsets $J\subseteq L_1$ gives the dimension of $\H_S(\bq)$ by Proposition~\ref{prop:dim}.
Hence the desired direct sum decomposition of $\H_S(\bq)$ holds.

Let $(I,\lambda)\in\Irr(\H_S(\bq))$. Since $\bP_{I,\lambda}^S$ is a direct summand of $\H_S(\bq)$, it is projective.
Lemma~\ref{lem:PS} implies
\[ \bP_{I,\lambda}^S \cong \bP_I^{S^{0,\lambda}} \otimes \bS_\lambda \cong \bP_I^{S^{0,\lambda}} \otimes \bS_\lambda^n \otimes\bS_\lambda^t. \]
Thus $\bP_{I,\lambda}^S$ can be viewed as a module over the algebra
$\H_{S^{0,\lambda}}(0)\bigotimes \left(\bigotimes_{j\in L_1^\lambda} \CC W_j \right)$.
This module is indecomposable with top isomorphic to $\bC_I^{S^{0,\lambda}} \otimes \bS_\lambda$ by Proposition~\ref{prop:TensorP} and its radical is $\rad\left(\bP_{I}^{S^{0,\lambda}} \right) \otimes \bS_\lambda$  by Proposition~\ref{prop:tensor}~(i), which is an $\H_S(\bq)$-submodule of $\bP_{I,\lambda}^S$ with $T_s$ acting by $0$ for all $s\in S^0\setminus S^{0,\lambda}$ and with $T_t$ acting by $1$ for all $t\in S_j$, $j\in L_1\setminus L_1^\lambda$. 
Then Proposition~\ref{prop:induce}~(i) implies that $\bP_{I,\lambda}^S$ is an indecomposable $\H_S(\bq)$-module, and Proposition~\ref{prop:induce}~(ii) implies that the top of $\bP_{I,\lambda}^S$ as an $\H_S(\bq)$-module is isomorphic to $\bC_I^{S^{0,\lambda}} \otimes \bS_\lambda$.
\end{proof}

\begin{corollary}\label{cor:RepLists}
The two sets $\left\{ \bP_{I,\lambda}: (I,\lambda) \in \Irr(\H_S(\bq) \right\}$ and $\left\{ \bC_{I,\lambda}: (I,\lambda) \in \Irr(\H_S(\bq) \right\}$ are, respectively, a complete list of non-isomorphic projective indecomposable $\H_S(\bq)$-modules and a complete list of non-isomorphic simple $\H_S(\bq)$-modules.
The Cartan matrix of $\H_S(\bq)$ is $\displaystyle\left[  c_{I,J}^{S^{0,\lambda}} \delta_{\lambda,\mu} \right]_{(I,\lambda),(J,\mu) \in\Irr(\H_S(\bq))}$.
\end{corollary}

\begin{proof}
Theorem~\ref{thm:decomp} implies that, every projective indecomposable [simple, resp.] $\H_S(\bq)$-module is isomorphic to $\bP_{I,\lambda}$ [$\bC_{I,\lambda}$, resp.] for some $(I,\lambda) \in \Irr(\H_S(\bq))$.
If there exist $(J,\mu)\in\Irr(\H_S(\bq))$ such that $\bC_{I,\lambda}^S \cong \bC_{J,\mu}^S$, then we have $\lambda = \mu$ by Proposition~\ref{prop:tensor} (ii), and this implies $I=J$ since $S^{0,\lambda} = S^{0,\mu}$.
Therefore if $(I,\lambda)\ne(J,\mu)$ then $\bC_{I,\lambda}^S\not\cong \bC_{J,\mu}^S$ and thus $\bP_{I,\lambda}^S \not\cong \bP_{J,\mu}^S$. 
Finally, using the Cartan matrices of $\H_S^0(\bq)\cong\H_{S^0}(0)$ and $\H_S^1(\bq)\cong \CC W^1$ we obtain the Cartan matrix of $\H_S(\bq)$.
\end{proof}

\begin{example}
(i) Let $m$ and $n$ be positive integers. The algebra $\H(0^m1^n)$ has the following decomposition
\[ \H(0^m1^n) \cong \left( \bigoplus_{\alpha\models m+1} \bP_\alpha \otimes \bS_{n+1} \right) \bigoplus 
\left( \bigoplus_{\alpha\models m} \bigoplus_{ \substack{ \lambda\vdash n+1\\ \lambda\ne n+1} } \left( \bP_\alpha\otimes\bS_\lambda\right)^{\oplus d_\lambda} \right).\]
We have $(\alpha,\lambda)\in\Irr(\H(0^m1^n))$ if and only if either $\alpha\models m+1$ and $\lambda=n+1$, or $\alpha\models m$, $\lambda\vdash n+1$, and $\lambda\ne n+1$.
The Cartan matrix of $\H(0^m1^n)$ is $\left[c_{\alpha,\beta} \cdot \delta_{\lambda,\mu} \right]_{(\alpha,\lambda),(\beta,\mu)\in\Irr(\H_S(\bq))}$.
\vskip5pt
\noindent(ii) Let $m_1,n_\ell\ge0$ and $n_1,m_2,\ldots,n_{\ell-1},m_\ell \ge1$ be integers.
The algebra $\H(0^{m_1}1^{n_1}\cdots0^{m_\ell}1^{n_\ell})$ has projective indecomposable representations and simple representations
\[ \bP_{\alpha^1,\lambda^1,\ldots,\alpha^\ell,\lambda^\ell}:= \bP_{\alpha^1}\otimes \bS_{\lambda^1}\otimes\cdots\otimes \bP_{\alpha^\ell}\otimes\bS_{\lambda^\ell} \qand
\bC_{\alpha^1,\lambda^1,\ldots,\alpha^\ell,\lambda^\ell}:= \bC_{\alpha^1}\otimes \bS_{\lambda^1}\otimes\cdots\otimes \bC_{\alpha^\ell}\otimes\bS_{\lambda^\ell} \]
indexed by tuples $(\alpha^1,\lambda^1,\ldots,\alpha^\ell,\lambda^\ell)$ satisfying $\lambda^i\vdash n_i+1$ and $\alpha^i \models m_i^\lambda+1$ for all $i\in[\ell]$, where 
\[ m_i^\lambda := \max\{0, m_i - \#\{ j\in\{i-1,i\}\cap[\ell]: \lambda^j\ne n_j \} \}. \]
\end{example}

\subsection{Quiver and representation type}
We first observe that the quiver of a 0-Hecke algebra is loopless.

\begin{lemma}\label{lem:H0QuiverLoopless}
The quiver of the 0-Hecke algebra of any finite Coxeter system is loopless.
\end{lemma}

\begin{proof}
Let $(W,S)$ be a finite Coxeter system in this proof.
Following Duchamp, Hivert, and Thibon~\cite[\S4.3]{NCSF-VI}, we only need to show that any short exact sequence of the form $0\to\bC_I^S\to \bM \to \bC_I^S\to 0$ must split for every $I\subseteq S$.
Since $\bC_I^S$ is one-dimensional, $\bM$ must be two-dimensional.
If every nonzero element of $\bM$ spans an $\H_S(0)$-submodule then we are done.
Otherwise there exist $u\in \bM$ and $s\in S$ such that $u$ and $v:= \pi_s(u)$ form a basis of $\bM$.
Then $\soc(\bM) = \rad(\bM)$ is the span of $v$, on which $\pi_s$ acts by one.
On the other hand, $\pi_s$ acts on $\top(\bM)$ by zero since $\pi_s(u) = v\in\rad(\bM)$.
Thus the short exact sequence $0\to\bC_I^S\to \bM \to \bC_I^S\to 0$ cannot hold.
\end{proof}

Now we are ready to describe the quiver of the algebra $\H_S(\bq)$.

\begin{proposition}\label{prop:quiver}
For each $\lambda\in \Irr(\CC W^1)$, the full subquiver $Q_{\lambda}(\bq)$ of the quiver of $\H_S(\bq)$ with all vertices of the form $\bC_{I,\lambda}^S$ (and with all arrows between these vertices in the quiver of $\H_S(\bq)$) is isomorphic to the quiver of the $0$-Hecke algebra $\H_{S^{0,\lambda}}(0)$, which is the tensor product of the quivers of the $0$-Hecke algebras generated by connected components of $S^{0,\lambda}$.
In addition, there is no arrow between $Q_{\lambda}(\bq)$ and $Q_{\mu}(\bq)$ if $\lambda,\mu\in\Irr(\CC W^1)$ are distinct.
\end{proposition}

\begin{proof}
Let $(I,\lambda)\in\Irr(\H_S(\bq))$.
Proposition~\ref{prop:tensor} (iii) implies that 
\[ \rad\left( \bP_{I,\lambda}^S \right) \Big/ \rad^2\left( \bP_{I,\lambda}^S \right) \cong 
\rad\left( \bP_I^{S^{0,\lambda}} \right) \Big/ \rad^2\left( \bP_I^{S^{0,\lambda}} \right) \otimes \bS_{\lambda}. \] 
Among the composition factors of this $\H_S(\bq)$-module, the multiplicity of a simple module $\bC_{J,\mu}^S$ with $(J,\mu)\in\Irr(\H_S(\bq))$ is either zero if $\lambda\ne\mu$, or equal to the multiplicity of $\bC_J^{S^{0,\lambda}}$ among the composition factors of $\rad \left( \bP_{I}^{S^{0,\lambda}} \right) \Big/ \rad^2 \left( \bP_{I}^{S^{0,\lambda}} \right)$ if $\lambda = \mu$. 
Therefore the full subquiver $Q_\lambda(\bq)$ is isomorphic to the quiver of the $0$-Hecke algebra $\H_{S^{0,\lambda}}(0)$, and there is no arrow between $Q_\lambda(\bq)$ and $Q_\mu(\bq)$ if $\lambda,\mu\in\Irr(\CC W^1)$ are distinct.
By Lemma~\ref{lem:H0QuiverLoopless},  the quivers of the $0$-Hecke algebras generated by connected components of $S^{0,\lambda}$ are all loopless, and the tensor product of these quivers gives the quiver of $\H_{S^{0,\lambda}}(0)$.
\end{proof}

An example will be given in the end of this section, after we determine the representation type of $\H_S(\bq)$ from its quiver.
Recall that Duchamp, Hivert, and Thibon~\cite[\S4.3]{NCSF-VI} constructed the quiver of the $0$-Hecke algebra $H_n(0)$ of type $A_{n-1}$ and showed that $H_n(0)$ is of finite representation type if and only if $n\le 3$.
In particular, the quiver of $\H_3(0)$ consists of three connected components, two of type $A_1$ and one of type $A_2$.
Using this observation we obtain the representation type of the $0$-Hecke algebra $\H_3(0)\otimes\H_3(0)$.

\begin{lemma}\label{lem:H3}
The algebra $\H_3(0)\otimes\H_3(0)$ is of infinite representation type.
\end{lemma}
\begin{proof}
The quiver of $\H_3(0)\otimes\H_3(0)$ is the tensor product of the quiver of $\H_3(0)$ and itself, and thus contains a cycle of length four as a connected component.
Orienting this cycle in such a way that it has no directed path of length two, one gets a quiver whose path algebra is isomorphic to a quotient of $\H_3(0)\otimes\H_3(0)$ and of infinite representation type (cf. Duchamp--Hivert--Thibon~\cite[\S4.3]{NCSF-VI}).
Combining this with Proposition~\ref{prop:surj} (iii) we conclude that $\H_3(0)\otimes\H_3(0)$ is of infinite representation type.
\end{proof}

Now we can determine the representation type of the algebra $\H_S(\bq)$.

\begin{proposition}
The algebra $\H_S(\bq)$ is of finite representation type if and only if $|S_i|\le 2$ for all $i\in L_0$ with equality occurring at most once.
\end{proposition}

\begin{proof}
Suppose $|S_i|\ge3$ for some $i\in L_0$.
Since $(W,S)$ is simply laced and $S_i$ is connected, there exists a subset $I\subseteq S_i$ which generates a Coxeter subsystem of type $A_3$.
The subalgebra of $\H_S(\bq)$ generated by the set $\{T_s: s\in I\}$ is isomorphic to the $0$-Hecke algebra $H_4(0)$, which is of infinite representation type by Duchamp--Hivert--Thibon~\cite[\S4.3]{NCSF-VI}.
This implies that $\H_S(\bq)$ is of infinite representation type, since every $\H_4(0)$-module becomes an $\H_S(\bq)$-module by letting all generators $T_s$ of $\H_S(\bq)$ with $s\in S\setminus I$ act by one and an indecomposable $\H_4(0)$-module is also an indecomposable $\H_S(\bq)$-module by Proposition~\ref{prop:induce} (i).

Next, assume $|S_i|=|S_{i'}|=2$ for distinct $i,i'\in L_0$.
Then $\H_3(0)\otimes\H_3(0)$ is a quotient of $\H_S(\bq)$.
It follows from Proposition~\ref{prop:surj} (iii) and  Lemma~\ref{lem:H3} that $\H_S(\bq)$ is of infinite representation type.

Finally, assume $|S_i|\le 2$ for all $i\in L_0$ with equality occurring at most once.
Since $\H_{m}(0)$ with $m\le 2$ and $\CC W_j$ with $j\in L_1$ are semisimple algebras, their quivers consist of isolated vertices.
By Duchamp--Hivert--Thibon~\cite[\S4.3]{NCSF-VI}, the quiver of $\H_3(0)$ consists of three connected components, two of type $A_1$ and one of type $A_2$.
Thus each connected component of the quiver of the algebra
\[ \H_S^0(0) \otimes \CC W^1 \cong \left( \bigotimes_{i\in L_0} \H_{S_i}(0) \right) \bigotimes \left( \bigotimes_{j\in L_1} \CC W_j \right) \]
is of type $A_1$ or $A_2$ by the definition of the tensor product of quivers.
Since $\H_S(\bq)$ is a quotient of the above algebra by the equation~\eqref{eq:Hq}, it follows from Proposition~\ref{prop:surj} (iii) that $\H_S(\bq)$ is of finite representation type.
\end{proof}

\begin{example}
By Proposition~\ref{prop:quiver}, the quiver of the algebra $\H(0^21^30^2)$ is the disjoint union of full subquivers $Q_\lambda$ indexed by partitions $\lambda$ of $4$.
If $\lambda=4$ then $Q_\lambda$ is the quiver of $\H_3(0)\otimes \H_3(0)$, which is the disjoint union of four isolated vertices, four paths of length two, and a cycle of length four.
If $\lambda\in\{(3,1), (2,2), (2,1,1), (1,1,1,1)\}$ then $Q_\lambda$ is the quiver of $\H_2(0)\otimes\H_2(0)$, which consists of four isolated vertices.
One sees that the algebra $\H(0^21^30^2)$ is of infinite representation type. 
\end{example}

\section{Induction and restriction}\label{sec:IndRes}

Let $\H_S(\bq)$ be the Hecke algebra of a finite simply-laced Coxeter system $(W,S)$ with independent parameters $\bq\in\{0,1\}^S$, and let $R\subseteq S$.
In this section we study the induction and restriction of representations between $\H_R(\bq)=\H_R(\bq|_R)$ and $\H_S(\bq)$, as there is an obvious algebra surjection from $\H_R(\bq)$ to the subalgebra of $\H_S(\bq)$ generated by $\{T_s:s\in R\}$ (which is not necessarily an isomorphism~\cite[\S3]{Hecke}).

By induction on $|R|$, we may assume $R=S\setminus\{s\}$ for some $s\in S$, without loss of generality.
We distinguish two cases ($q_s=0$ and $q_s=1$) in the next two subsections.
In each case our results exhibit a two-sided duality, i.e., both adjunctions \eqref{eq:dual1} and \eqref{eq:dual2} are true. 

\subsection{Case 1}\label{sec:IndRes1}
In this subsection we study the case $R=S\setminus\{s\}$ for some $s\in S$ with $q_s=0$. 
One sees that $R^0=S^0\setminus\{s\}$ and $R^1=S^1$. 
We first study induction from $\H_R(\bq)$ to $\H_S(\bq)$.

\begin{proposition}\label{prop:Ind1}
Suppose $R=S\setminus\{s\}$ for some $s\in S$ with $q_s=0$.
Let $(I,\lambda)\in\Irr(\H_R(\bq))$. Then 
\begin{align*}
\bP_{I,\lambda}^R \uparrow\,_{\H_R(\bq)}^{\H_S(\bq)} &\cong 
\begin{cases}
\bP_{I,\lambda}^S & \text{ if } s\notin S^{0,\lambda}, \\
\bP_{I,\lambda}^S \oplus \bP_{I\cup\{s\},\lambda}^S & \text{ if } s\in S^{0,\lambda}
\end{cases} 
\end{align*}
where each $\H_S(\bq)$-module on the right hand side is projective indecomposable.
Furthermore, if $w$ is any element of $W_{R^{0,\lambda}}$ with $D(w)=I$, then we have the following equality
\[ \bC_{I,\lambda}^R \uparrow\,_{\H_R(\bq)}^{\H_S(\bq)} =
\sum_{ \substack{ z\in W_{S^{0,\lambda}} \\ D(z^{-1})\subseteq\{s\} } } \bC_{D(wz),\lambda}^S \]
in the Grothendieck group $G_0(\H_S(\bq))$, where each $\H_S(\bq)$-module on the right hand side is simple.
\end{proposition}

\begin{proof}
By the structure~\eqref{eq:Hq} of the algebra $\H_S(\bq)$ and Lemma~\ref{lem:PS}, we have
\[ \bP_{I,\lambda}^R \uparrow\,_{\H_R(\bq)}^{\H_S(\bq)} \cong  
\left( \bP_I^{R^{0,\lambda}} \otimes \bS_\lambda \right) \uparrow\,_{\H_R(\bq)}^{\H_S(\bq)}
\cong 
\left( \bP_I^{R^{0,\lambda}} \uparrow\,_{\H_{R^{0,\lambda}}(0)}^{\H_{S^{0,\lambda}}(0)} \right) \otimes \bS_\lambda. \]

First assume $s\notin S^{0,\lambda}$. 
Then we have $S^{0,\lambda} = R^{0,\lambda}$ which implies $(I,\lambda)\in \Irr(\H_S(\bq))$, and 
\[ \bP_{I,\lambda}^R \uparrow\,_{\H_R(\bq)}^{\H_S(\bq)} \cong  
\bP_I^{S^{0,\lambda}} \otimes \bS_\lambda \cong \bP_{I,\lambda}^S. \]

Next assume $s\in S^{0,\lambda}$. 
Then $S^{0,\lambda} = R^{0,\lambda} \sqcup \{s\}$, which implies that $(I,\lambda)$ and $(I\cup\{s\},\lambda)$ are both in $\Irr(\H_S(\bq))$. 
By the induction formula~\eqref{eq:H0Ind} for $0$-Hecke modules, we have
\[ \bP_{I,\lambda}^R \uparrow\,_{\H_R(\bq)}^{\H_S(\bq)} \cong \left( \bP_I^{S^{0,\lambda}} \oplus \bP_{I\cup\{s\}}^{S^{0,\lambda}} \right) \otimes \bS_\lambda 
\cong \bP_{I,\lambda}^S  \oplus \bP_{I\cup\{s\},\lambda}^S. \]

Now let $w$ be an element of $W_{R^{0,\lambda}}$ with $D(w)=I$. 
By the induction formula~\eqref{eq:H0Ind} for $0$-Hecke modules,
\begin{align*}
\bC_{I,\lambda}^R \uparrow\,_{\H_R(\bq)}^{\H_S(\bq)} & \cong 
\left( \bC_I^{R^{0,\lambda}} \otimes \bS_\lambda \right) \uparrow\,_{\H_R(\bq)}^{\H_S(\bq)} \\
& \cong \left( \bC_I^{R^{0,\lambda}} \uparrow\,_{\H_{R^{0,\lambda}}(0)}^{\H_{S^{0,\lambda}}(0)} \right) \otimes \bS_\lambda  \\
& = \sum_{ \substack{ z\in W_{S^{0,\lambda}} \\ D(z^{-1})\subseteq\{s\} }} \bC_{D(wz)}^{S^{0,\lambda}} \otimes \bS_\lambda
\end{align*}
where the last sum holds in the Grothendieck group $G_0(\H_S(\bq))$.
Each summand $\bC_{D(wz)}^{S^{0,\lambda}} \otimes \bS_\lambda$ is isomorphic to the simple $\H_S(\bq)$-module indexed by $(D(wz),\lambda)\in\Irr(\H_S(\bq))$ since $D(wz)\subseteq S^{0,\lambda}$. 
\end{proof}

\begin{example}
(i) Let $\bq=(0^2,1^1,0^3,1^2)$, $\bq_1=(0^2,1^1,0^2)$ and $\bq_2=(0^0,1^2)$.
Then
\[ \left(\bP_{(3),[2],(1,2)} \otimes \bP_{(1),[2,1]} \right) \uparrow\,_{\H(\bq_1)\otimes\H(\bq_2)}^{\H(\bq)} 
\cong \bP_{(3), [2],(1,2),[2,1]} \qand \]
\[ \left(\bP_{(3),[2],(1,2)} \otimes \bP_{(1),[3]} \right) \uparrow\,_{\H(\bq_1)\otimes\H(\bq_2)}^{\H(\bq)} 
\cong \bP_{(3),[2],(1,3),[3]} \oplus \bP_{(3),[2],(1,2,1),[3]}. \]
(ii) Let $\bq=(0^2,1^1,0^3,1^2)$, $\bq_1=(0^2,1^1,0^1)$ and $\bq_2=(0^1,1^2)$.
Since $21\shuffle 1 = \{213,231,321\}$, we have
\[ \left(\bC_{(3), [2],(1,1)} \otimes \bC_{(1),[2,1]} \right) \uparrow\,_{\H(\bq_1)\otimes\H(\bq_2)}^{\H(\bq)} 
\cong \bC_{(3), [2], (1,2), [2,1]} \oplus \bC_{(3), [2], (2,1), [2,1]} \oplus \bC_{(3), [2], (1,1,1), [2,1]}.\]
\end{example}

Now we study restriction of $\H_S(\bq)$-modules to $\H_R(\bq)$.

\begin{proposition}\label{prop:Res1}
Suppose $R=S\setminus\{s\}$ for some $s\in S$ with $q_s=0$.
Let $(I,\lambda)\in\Irr(\H_S(\bq))$. Then 
\[ \bP_{I,\lambda}^S \downarrow\,_{\H_R(\bq)}^{\H_S(\bq)} \cong \bigoplus_{K\in\, I\,\downarrow\,_{R^{0,\lambda}}^{S^{0,\lambda}}} \bP_{K,\lambda}^R \]
where each direct summand is a projective indecomposable $\H_R(\bq)$-module.
Furthermore, we have 
\[ \bC_{I,\lambda}^S \downarrow\,_{\H_R(\bq)}^{\H_S(\bq)} \cong \bC_{I\cap R^{0,\lambda},\lambda}^R \]
where the right hand side is a simple $\H_R(\bq)$-module. 
\end{proposition}

\begin{proof}
By the structure~\eqref{eq:Hq} of $\H_S(\bq)$ and the restriction formula~\eqref{eq:H0Res} for $0$-Hecke modules, we have
\begin{align*}
\bP_{I,\lambda}^S \downarrow\,_{\H_R(\bq)}^{\H_S(\bq)} 
& \cong \left( \bP_I^{S^{0,\lambda}} \otimes \bS_\lambda \right) \downarrow\,_{\H_R(\bq)}^{\H_S(\bq)} \\
& \cong \left( \bP_I^{S^{0,\lambda}} \downarrow\,_{\H_{R^{0,\lambda}}(0)}^{\H_{S^{0,\lambda}}(0)} \right) \otimes \bS_\lambda \\
& \cong \bigoplus_{K\in\, I\,\downarrow\,_{R^{0,\lambda}}^{S^{0,\lambda}}} \bP_K^{R^{0,\lambda}} \otimes \bS_\lambda.
\end{align*}
For each $K\in I\,\downarrow\,_{R^{0,\lambda}}^{S^{0,\lambda}}$, one sees that $(K,\lambda)\in\Irr(\H_R(\bq))$ since $K\subseteq R^{0,\lambda}$, and that $\bP_K^{R^{0,\lambda}} \otimes \bS_\lambda$ is isomorphic to the projective indecomposable $\H_R(\bq)$-module $\bP_{K,\lambda}^R$ by Lemma~\ref{lem:PS}. 

Similarly we have
\begin{align*}
\bC_{I,\lambda}^S \downarrow\,_{\H_R(\bq)}^{\H_S(\bq)} 
& \cong \left( \bC_I^{S^{0,\lambda}} \otimes \bS_\lambda \right) \downarrow\,_{\H_R(\bq)}^{\H_S(\bq)} \\
& \cong \left( \bC_I^{S^{0,\lambda}} \downarrow\,_{\H_{R^{0,\lambda}}(0)}^{\H_{S^{0,\lambda}}(0)} \right) \otimes \bS_\lambda \\
& \cong \bC_{I\cap R^{0,\lambda}}^{R^{0,\lambda}} \otimes \bS_\lambda
\end{align*}
where the last term is isomorphic to the simple $\H_R(\bq)$-module indexed by $(I\cap R^{0,\lambda},\lambda)\in\Irr(\H_R(\bq))$.
\end{proof}

\begin{example}
Let $\bq=(0^2,1^2,0^4,1^3)$, $\bq_1=(0^2,1^2,0^2)$, and $\bq_2=(0^1,1^3)$. 
We have 
\[ \bP_{(3),[2,1],(1,2),[2,1,1]} \downarrow\,_{\H(\bq_1)\otimes\H(\bq_2)}^{\H(\bq)} 
\cong \left( \bP_{(3),[2,1],(1,1)} \otimes \bP_{(1),[2,1,1]} \right)
\oplus \left( \bP_{(3),[2,1],(2)} \otimes \bP_{(1),[2,1,1]} \right) \]
since $(1,2)\downarrow_2\ = \{ ((1,1),(1)), ((2),(1)) \}$~\cite[Proposition~4.5]{H0Tab}.
We also have 
\[ \bC_{(3),[2,1],(1,2),[2,1,1]} \downarrow\,_{\H(\bq_1)\otimes\H(\bq_2)}^{\H(\bq)} 
\cong \bC_{(3),[2,1],(1,1)} \otimes \bC_{(1),[2,1,1]} \]
since $(1,2)_{\le 2}=(1,1)$ and $(1,2)_{>2}=(1)$.
\end{example}

\begin{corollary}\label{cor:dual1}
Suppose $R=S\setminus\{s\}$ for some $s\in S$ with $q_s=0$.
If $(I,\lambda)\in\Irr(\H_R(\bq))$ and $(J,\mu)\in\Irr(\H_S(\bq))$ then
\[ \left\langle \bP_{I,\lambda}^R \uparrow\,_{\H_R(\bq)}^{\H_S(\bq)},\ \bC_{J,\mu} \right\rangle
= \left\langle \bP_{I,\lambda}^R,\ \bC_{J,\mu} \downarrow\,_{\H_R(\bq)}^{\H_S(\bq)} \right\rangle, \]
\[ \left\langle \bP_{J,\mu}^S \downarrow\,_{\H_R(\bq)}^{\H_S(\bq)},\ \bC_{I,\lambda} \right\rangle
= \left\langle \bP_{J,\mu}^S,\ \bC_{I,\lambda} \uparrow\,_{\H_R(\bq)}^{\H_S(\bq)} \right\rangle, \]
\end{corollary}

\begin{proof}
This follows from Proposition~\ref{prop:Ind1}, Proposition~\ref{prop:Res1}, and the two-sided duality~\eqref{eq:H0Duality} for $0$-Hecke modules.
\end{proof}

\subsection{Case 2}
Now we study the case $R=S\setminus\{t\}$ for some $t\in S$ with $q_t=1$.
One sees that $R^0=S^0$ and $R^1=S^1\setminus\{t\}$. 
We will also need the following lemma.

\begin{lemma}\label{lem:RS}
Suppose $c_\mu^\lambda\ne0$ for some $\lambda\in\Irr(\H_S(\bq))$ and some $\mu\in\Irr(\H_R(\bq))$.
Then $S^{0,\lambda}\subseteq R^{0,\mu}$.
\end{lemma}

\begin{proof}
Suppose $j\in L_1^\mu$, that is, $W_j$ acts nontrivially on $\bS_\mu$ for some $j\in L_1$.
Since $c_\mu^\lambda\ne0$, we have $W_j$ acts nontrivially on $\bS_\lambda$ by Lemma~\ref{lem:RestrictOne}, i.e., $j\in L_1^\lambda$.
Thus $L_1^\mu\subseteq L_1^\lambda$, which implies $S^{0,\lambda}\subseteq R^{0,\mu}$.
\end{proof}

We are ready to give the formulas for induction of $\H_R(\bq)$-modules to $\H_S(\bq)$.

\begin{proposition}\label{prop:Ind2}
Suppose $R=S\setminus\{t\}$ for some $t\in S$ with $q_t=1$.
Let $(J,\mu)\in\Irr(\H_R(\bq))$. Then
\[ \bP_{J,\mu}^R \uparrow\,_{\H_R(\bq)}^{\H_S(\bq)} \cong 
\bigoplus_{ \lambda\in\Irr(\CC W^1):\ J\subseteq S^{0,\lambda} } \left( \bP_{J,\lambda}^S \right)^{\oplus c_\mu^\lambda}, \]
\[ \bC_{J,\mu}^R \uparrow\,_{\H_R(\bq)}^{\H_S(\bq)} \cong 
\bigoplus_{ \lambda\in\Irr(\CC W^1):\ J\subseteq S^{0,\lambda} } \left( \bC_{J,\lambda}^S \right)^{\oplus c_\mu^\lambda} \]
where each summand $\bP_{J,\lambda}^S$ [or $\bC_{J,\lambda}^S$, resp.] with multiplicity $c_\mu^\lambda\ne0$ is a projective indecomposable [or simple, resp.] $\H_S(\bq)$-module indexed by $(J,\lambda)\in\Irr(\H_S(\bq))$.
\end{proposition}

\begin{proof}
By the structure~\eqref{eq:Hq} of the algebra $\H_S(\bq)$, we have
\[ \bP_{J,\mu}^R \uparrow\,_{\H_R(\bq)}^{\H_S(\bq)} \cong
\bP_J^{R^{0,\mu}} \otimes \left( \bS_\mu \uparrow\,_{W_{R^1}}^{W_{S^1}} \right) \qand 
\bC_{J,\mu}^R \uparrow\,_{\H_R(\bq)}^{\H_S(\bq)} \cong
\bC_J^{R^{0,\mu}} \otimes \left( \bS_\mu \uparrow\,_{W_{R^1}}^{W_{S^1}} \right). \]
By the induction formula~\eqref{eq:IndResG}, we have
\[ \bS_\mu \uparrow\,_{W_{R^1}}^{W_{S^1}} \cong \bigoplus_{\lambda\in\Irr(\CC W^1)} \bS_\lambda^{\oplus c_\mu^\lambda}. \]
Let $\lambda\in\Irr(\CC W^1)$ with $c_\mu^\lambda\ne0$.
Then $S^{0,\lambda}\subseteq R^{0,\mu}$ by Lemma~\ref{lem:RS}.
Using a similar argument to the proof of Lemma~\ref{lem:PS}, one sees that $\bP_J^{R^{0,\mu}}\otimes \bS_\lambda=0$ and $\bC_J^{R^{0,\mu}}\otimes \bS_\lambda=0$ if $J\not\subseteq S^{0,\lambda}$, or $\bP_J^{R^{0,\mu}}\otimes \bS_\lambda \cong \bP_{J,\lambda}^S$ and $\bC_J^{R^{0,\mu}}\otimes \bS_\lambda \cong \bC_{J,\lambda}^S$ with $(J,\lambda)\in\Irr(\H_S(\bq))$ if $J\subseteq S^{0,\lambda}$.
\end{proof}

\begin{example}
Let $\bq=(0^3,1^2,0^1)$, $\bq_1=(0^3,1^1)$, and $\bq_2=(1^0,0^1)$.
By the Littlewood-Richardson Rule, $c^{[3]}_{[2],[1]} = c^{[2,1]}_{[2],[1]} = 1$ and $c^\lambda_{[2],[1]}=0$ for any partition $\lambda$ different from $[3]$ and $[2,1]$. Thus
\begin{eqnarray*}
\left( \bP_{(3,1), [2]} \otimes \bP_{[1], (2)} \right) \uparrow\,_{\H(\bq_1)\otimes\H(\bq_2)}^{\H(\bq)} 
&\cong& \bP_{(3,1),[3],(2)} \qand \\
\left( \bP_{(1,3), [2]} \otimes \bP_{[1], (2)} \right) \uparrow\,_{\H(\bq_1)\otimes\H(\bq_2)}^{\H(\bq)} 
&\cong& \bP_{(1,3),[3],(2)} \oplus \bP_{(1,2),[2,1],(1)}.
\end{eqnarray*}
The same result holds if $\bP$ is replaced with $\bC$.
\end{example}

Next, we study restriction of $\H_S(\bq)$-modules to $\H_R(\bq)$.

\begin{proposition}\label{prop:Res2}
Suppose $R=S\setminus\{t\}$ for some $t\in S$ with $q_t=1$.
Let $(I,\lambda)\in\Irr(\H_S(\bq))$. Then
\[ \bC_{I,\lambda}^S \downarrow\,_{\H_R(\bq)}^{\H_S(\bq)} \cong 
\bigoplus_{\mu\in\Irr(\CC W_{R^1})} \left( \bC_{I,\mu}^R \right)^{\oplus c_\mu^\lambda} \]
where each summand $\bC_{I,\mu}^R$ with multiplicity $c_\mu^\lambda\ne0$ is a simple $\H_R(\bq)$-module indexed by $(I,\mu)\in\Irr(\H_R(\bq))$, and
\[ \bP_{I,\lambda}^S \downarrow\,_{\H_R(\bq)}^{\H_S(\bq)} \cong 
\bigoplus_{\mu\in\Irr(\CC W_{R^1})} \bQ_{I,S^{0,\lambda}}^{R^{0,\mu}} \otimes \bS_\mu^{\oplus c_\mu^\lambda} \]
where each summand $\bQ_{I,S^{0,\lambda}}^{R^{0,\lambda}} \otimes \bS_\mu$ with multiplicity $c_\mu^\lambda\ne0$ is an indecomposable $\H_R(\bq)$-module with top isomorphic to the simple $\H_R(\bq)$-module $\bC_{I,\mu}^R$ and is projective if and only if $S^{0,\lambda} = R^{0,\mu}$.
(See Section~\ref{sec:H0} for the definition of the $0$-Hecke module $\bQ_{I,S^{0,\lambda}}^{R^{0,\lambda}}$.)
\end{proposition}

\begin{proof}
By Lemma~\ref{lem:PS}, we have
\[ \bP_{I,\lambda}^S  \cong \bP_I^{S^{0,\lambda}}\otimes \bS_\lambda \qand
\bC_{I,\lambda}^S  \cong \bC_I^{S^{0,\lambda}}\otimes \bS_\lambda \]
where $\pi_s$ acts by zero for all $s\in S\setminus S^{0,\lambda}$.
Applying the restriction formula~\eqref{eq:IndResG} gives
\[ \bP_{I,\lambda}^S \downarrow\,_{\H_R(\bq)}^{\H_S(\bq)} \cong 
\bP_I^{S^{0,\lambda}} \otimes \left( \bS_\lambda \downarrow\,_{\CC W_{R^1}}^{\CC W_{S^1}} \right) 
\cong \bigoplus_{\mu\in\Irr(\CC W_{R^1})} \bP_I^{S^{0,\lambda}} \otimes \bS_\mu^{\oplus c_\mu^\lambda}, \]
\[ \bC_{I,\lambda}^S \downarrow\,_{\H_R(\bq)}^{\H_S(\bq)} \cong 
\bC_I^{S^{0,\lambda}} \otimes \left( \bS_\lambda \downarrow\,_{\CC W_{R^1}}^{\CC W_{S^1}} \right) 
\cong \bigoplus_{\mu\in\Irr(\CC W_{R^1})} \bC_I^{S^{0,\lambda}} \otimes \bS_\mu^{\oplus c_\mu^\lambda}. \]
Let $\mu\in\Irr(\CC W_{R^1})$ with $c_\mu^\lambda\ne0$.
We have $S^{0,\lambda}\subseteq R^{0,\mu}$ by Lemma~\ref{lem:RS}.
Hence $(I,\mu)\in\Irr(\H_R(\bq))$ and
\[ \bC_I^{S^{0,\lambda}} \otimes \bS_\mu \cong \bC_I^{R^{0,\mu}} \otimes \bS_\mu = \bC_{I,\mu}^R. \]
Using Proposition~\ref{prop:TensorP} and Proposition~\ref{prop:induce} (i) one can show that $\bP_I^{S^{0,\lambda}} \otimes \bS_\mu$ is an indecomposable $\H_R(\bq)$-module. 
It follows from Lemma~\ref{lem:QIJ} (with $J=S^{0,\lambda}$ and $S=R^{0,\mu}$) that
\[ \bP_I^{S^{0,\lambda}} \otimes \bS_\mu \cong \bQ_{I,S^{0,\lambda}}^{R^{0,\mu}} \otimes \bS_\mu \qand
\top\left( \bQ_{I,S^{0,\lambda}}^{R^{0,\mu}} \otimes \bS_\mu \right) \cong
\bC_I^{R^{0,\mu}} \otimes \bS_\mu = \bC_{I,\mu}^R. \]
Thus $\bQ_{I,S^{0,\lambda}}^{R^{0,\mu}} \otimes \bS_\mu$ is projective if and only if it is isomorphic to $\bP_{I,\mu}^S$, which is equivalent to $S^{0,\lambda} = R^{0,\mu}$ by Lemma~\ref{lem:QIJ}.
\end{proof}

\begin{example}
Let $\bq=(0^3,1^2,0^1)$, $\bq_1=(0^3,1^1)$, and $\bq_2=(1^0,0^1)$.
By the Littlewood-Richardson Rule, we have $c^{[2,1]}_{[2],[1]} = c^{[2,1]}_{[1,1],[1]} = 1$ and $c^{[2,1]}_{\mu,\nu}=0$ for all $(\mu,\nu)\notin \{ ([2],[1]), ([1,1],[1]) \}$. Thus
\[ \bP_{(1,2),[2,1],(1)} \downarrow\,_{\H(\bq_1)\otimes\H(\bq_2)}^{\H(\bq)} 
\cong \left( \bP_{(1,2),[2]} \otimes \bP_{[1],(1)} \right) \oplus \left( \bP_{(1,2),[1,1]} \otimes \bP_{[1],(1)} \right). \]
Here $\bP_{(1,2),[2]}$ is isomorphic to $\bQ_{(1,3)}^{(3,1)}\otimes \bS_{[2]}$ (cf. Figure~\ref{fig:Q-module}), a nonprojective indecomposable $\H(\bq_1)$-module. 
On the other hand, $\bP_{(1,2),[1,1]}$ and $\bP_{[1],(1)}$ are projective indecomposable modules over $\H(\bq_1)$ and $\H(\bq_2)$, respectively.
We also have 
\[ \bC_{(1,2),[2,1],(1)} \downarrow\,_{\H(\bq_1)\otimes\H(\bq_2)}^{\H(\bq)} 
\cong \left( \bC_{(1,3),[2]} \otimes \bC_{[1],(2)} \right) \oplus \left( \bC_{(1,2),[1,1]} \otimes \bC_{[1],(2)} \right). \]
\end{example}

\begin{corollary}
Suppose $R=S\setminus\{t\}$ for some $t\in S$ with $q_t=1$.
If $(I,\lambda)\in \Irr(\H_S(\bq))$ and $(J,\mu)\in \Irr(\H_R(\bq))$ then
\[ \left\langle \bP_{I,\lambda}^S \downarrow\,_{\H_R(\bq)}^{\H_S(\bq)},\ \bC_{J,\mu}^R \right\rangle = 
\left\langle \bP_{I,\lambda}^S,\ \bC_{J,\mu}^R\uparrow\,_{\H_R(\bq)}^{\H_S(\bq)} \right\rangle, \]
\[ \left\langle \bP_{J,\mu}^R \uparrow\,_{\H_R(\bq)}^{\H_S(\bq)},\ \bC_{I,\lambda}^S \right\rangle = 
\left\langle \bP_{J,\mu}^R,\ \bC_{I,\lambda}^S\downarrow\,_{\H_R(\bq)}^{\H_S(\bq)} \right\rangle. \]
\end{corollary}

\begin{proof}
The result follows from Proposition~\ref{prop:Ind2}, Proposition~\ref{prop:Res2}, and the equations~\eqref{eq:hom} and \eqref{eq:H0Duality}.
\end{proof}

\section{Final remarks and questions}\label{sec:remark}

\subsection{Hecke algebras at roots of unity}
Let $\H_n(q)$ be a Hecke algebra of type $A_{n-1}$ over a field $\FF$ of characteristic zero with a single parameter $q\ne0$.
For each $\lambda\vdash n$, Dipper and James~\cite{DJ} constructed an $\H_n(q)$-module $\bS_\lambda(q)$, called the \emph{Specht module}, whose dimension equals the number $d_\lambda$ of standard Young tableaux of shape $\lambda$.
If $q$ is not zero or a root of unity then $\{\bS_\lambda(q):\lambda\vdash n\}$ is a complete set of non-isomorphic simple $\H_n(q)$-modules.
When $q$ is a primitive $k$th root of unity, Dipper and James~\cite{DJ} also constructed a complete set of simple $\H_n(q)$-modules $\mathbf{D}_\mu(q)$, where $\mu$ runs through all partitions of $n$ with at most $k-1$ rows of equal length.
However, these modules are not completely understood yet.

In this paper we study the (complex) representation theory of the Hecke algebra $\H_S(\bq)$ of a finite simply-laced Coxeter system $(W,S)$ with independent parameters $\bq\in\left( \CC\setminus\{\text{roots of unity}\} \right)^S$.
A natural question to ask is, whether our results can be extended to the case when the parameters are allowed to be roots of unity.


\subsection{Monoid algebra}
The Hecke algebra $\H_n(q)$ is a group algebra when $q=1$ or a monoid algebra when $q=0$.
The representation theory of finite groups is of course well known. 
The representation theory of finite monoids has also been widely studied; see, e.g., Steinberg~\cite{Monoid}.
In fact, the representation theory of $0$-Hecke algebras is a special case of the representation theory of $\mathcal J$-trivial monoids studied by Denton, Hivert, Schilling, and Thi\'ery~\cite{J}.

By Proposition~\ref{prop:q->1}, to study the Hecke algebra $\H_S(\bq)$ with $\bq\in\left( \CC\setminus\{\text{roots of unity}\} \right)^S$, we may assume $\bq\in\{0,1\}^S$, without loss of generality.
Then $\H_S(\bq)$ becomes a monoid algebra, although the underlying monoid is not $\J$-trivial (nor $\mathcal R$-trivial). 
Nevertheless, it may still be possible to recover our results via the representation theory of finite monoids and this is worth further investigation.

\subsection{The Grothendieck groups of type A Hecke algebras}
For $n\ge1$ we define
\[ G_0^n := \bigoplus_{\bq\in\{0,1\}^{n-1}} G_0(\H(\bq)) \qand
K_0^n := \bigoplus_{\bq\in\{0,1\}^{n-1}} K_0(\H(\bq)). \]
For $n=0$ we set $G_0^n := G_0(\CC)$ and $K_0^n:=K_0(\CC)$.
We can define algebra and coalgebra structures on the two Grothendieck groups 
\[ G_0 :=  \bigoplus_{n\ge0} G_0^n \qand K_0(\H) := \bigoplus_{n\ge0} K_0^n. \]

If $M$ is a (projective) $\H(\bq_1)$-module, where $\bq_1\in\{0,1\}^{m-1}$, and if $N$ is a (projective) $\H(\bq_2)$-module, where $\bq_2\in\{0,1\}^{n-1}$, then we define 
\[ M\htimes N := \left( M \otimes N \right) \uparrow\,_{\H(\bq_1)\otimes\H(\bq_2)}^{\H(\bq)} \]
where $\bq:=\bq_10\bq_2\in\{0,1\}^{m+n-1}$ is the concatenation of $\bq_1$, $0$, and $\bq_2$.
Also set $\bS_\emptyset\htimes N := N$ and $M\htimes \bS_\emptyset := M$, where $\bS_\emptyset$ for the unique simple $\CC$-module.

If $M$ is a (projective) $\H(\bq)$-module, where $\bq=(q_1,\ldots,q_{m-1})\in\{0,1\}^{m-1}$, then we define
\[\Delta M := \bS_\emptyset\otimes M + \sum_{i\in[m-1]:\, q_i=0} M \downarrow\,_{\H(\bq_{<i})\otimes\H(\bq_{>i})}^{\H(\bq)} + M\otimes\bS_\emptyset \]
where $\bq_{<i}:=(q_1,\ldots,q_{i-1})$ and $\bq_{>i}:=(q_{i+1},\ldots,q_{m-1})$.
 
\begin{proposition}\label{prop:GK}
With $\htimes$ and $\Delta$, the Grothendieck groups $G_0$ and $K_0$ become dual graded algebras and coalgebras.
\end{proposition} 

\begin{proof}
Let $\bq_1\in\{0,1\}^{m-1}$, $\bq_2\in\{0,1\}^{n-1}$, and $\bq=\bq_10\bq_2$.
The decomposition of $\H(\bq)$ given by Theorem~\ref{thm:decomp} and the restriction formulas for projective indecomposable modules given by Proposition~\ref{prop:Res1} imply that $\H(\bq)$ is a left projective module over $\H(\bq_1)\otimes\H(\bq_2)$.
One sees that $\H_S(\bq)$ is isomorphic to its opposite algebra $\H_S(\bq)^{\op}$
by sending $T_{s_1}\cdots T_{s_k}$ to $T_{s_k}\cdots T_{s_1}$ for all $s_1,\ldots,s_k\in S$.
Thus $\H(\bq)$ is also a right projective module over $\H(\bq_1)\otimes\H(\bq_2)$.
Then using a similar argument as Bergeron and Li~\cite[\S3]{BergeronLi} we can show that $\htimes$ and $\Delta$ are well-defined product and coproduct for $G_0$ and $K_0$. 
The duality follows from Corollary~\ref{cor:dual1}.
\end{proof}

One can check that $\Delta \bC_{(1),[2]} \htimes \Delta \bC_{(2)}$ contains the term 
\[ \left( \bS_\emptyset\otimes \bC_{(1),[2]} \right) \htimes \left( \bC_{(2)}\otimes \bS_\emptyset\right) = \bC_{(2)}\otimes \bC_{(1),[2]}.\]
But by Proposition~\ref{prop:Ind1} and Proposition~\ref{prop:Res1}, this term does not appear in
\[ \Delta\left( \bC_{(1),[2]} \htimes \bC_{(2)} \right)
= \Delta\left( \bC_{(1),[2],(1)} \htimes \bC_{(2)} \right) 
=  \Delta \left( \bC_{(1),[2],(3)} + \bC_{(1),[2],(2,1)} + \bC_{(1),[2],(1,2)} \right). \]
Thus $G_0$ is not a bialgebra.
By duality, $K_0$ is not a bialgebra either.

Although $G_0$ and $K_0$ are not bialgebras, they still have well-defined antipode maps since they are simultaneously an algebra and a coalgebra~\cite[\S7.3]{Hecke}.
It would be interesting to see whether the antipodes of $G_0$ and $K_0$ have simple formulas analogous to the antipode formulas for the Hopf algebras $G_0(\CC\SS_*)$, $G_0(\H_*(0))$, and $K_0(\H_*(0))$.
In addition, as these Hopf algebras correspond to $\Sym$, $\QSym$, and $\NSym$, one may try to find similar correspondences from $G_0$ and $K_0$ to some generalizations of symmetric functions.

\section*{acknowledgements}
We thank the anonymous referee for helpful suggestions and comments on this paper.



\begin{thebibliography}{}
%
%

\bibitem{ABR1}
R. M. Adin, F. Brenti\ and\ Y. Roichman, A unified construction of Coxeter group representations, Adv. in Appl. Math. {\bf 37} (2006), no.~1, 31--67.

\bibitem{ABR2}
R. M. Adin, F. Brenti\ and\ Y. Roichman, A construction of Coxeter group representations. II, J. Algebra {\bf 306} (2006), no.~1, 208--226.

\bibitem{ASS}
I. Assem, D. Simson, and A. Skowro\'nski, {\it Elements of the representation theory of associative algebras, vol. 1: Techniques of representation theory}, London Mathematical Society Student Texts, vol. 65, Cambridge University Press, Cambridge, 2006.

\bibitem{BergeronLi}
N. Bergeron\ and\ H. Li, Algebraic structures on Grothendieck groups of a tower of algebras, J. Algebra {\bf 321} (2009), no.~8, 2068--2084. 

\bibitem{BjornerBrenti}
A. Bj\"orner and F. Brenti, {\it Combinatorics of Coxeter groups}, Graduate Texts in Mathematics, 231, Springer, New York, 2005. 

\bibitem{BjornerWachs}
A. Bj\"orner and M. Wachs, Generalized quotients in Coxeter groups, Trans. Amer. Math. Soc. {\bf 308} (1988), no.~1, 1--37. 

\bibitem{CurtisReiner}
C. W. Curtis\ and\ I. Reiner, {\it Methods of representation theory. Vol. I}, John Wiley \&\ Sons, Inc., New York, 1981. 

\bibitem{J}
T. Denton, F. Hivert, A. Schilling, and N. Thi{\'e}ry, On the representation theory of finite $\mathcal J$-trivial monoids, S\'em. Lothar. Combin. B64d (2011), 44 pp. 

\bibitem{DJ}
R. Dipper\ and\ G. James, Representations of Hecke algebras of general linear groups, Proc. London Math. Soc. (3) {\bf 52} (1986), no.~1, 20--52. 

\bibitem{NCSF-VI}
G. Duchamp, F. Hivert\ and\ J.-Y. Thibon, Noncommutative symmetric functions. VI. Free quasi-symmetric functions and related algebras, Internat. J. Algebra Comput. {\bf 12} (2002), no.~5, 671--717. 

\bibitem{Etingof}
P. Etingof, O. Golberg, S. Hensel, T. Liu, A. Schwendner, D. Vaintrob, and E. Yudovina, Introduction to representation theory, Lecture notes retrieved from \href{https://arxiv.org/abs/0901.0827v5}{arXiv:0901.0827v5}.

\bibitem{HeckeRoot}
M. Geck\ and\ N. Jacon, {\it Representations of Hecke algebras at roots of unity}, Algebra and Applications, 15, Springer-Verlag London, Ltd., London, 2011. 

\bibitem{HeckeRootA}
F. M. Goodman\ and\ H. Wenzl, Iwahori-Hecke algebras of type $A$ at roots of unity, J. Algebra {\bf 215} (1999), no.~2, 694--734. 

\bibitem{GrinbergReiner}
D. Grinberg and V. Reiner, Hopf algebras in Combinatorics, Lecture notes retrieved from \href{https://arxiv.org/abs/1409.8356v4}{arXiv:1409.8356v4}.

\bibitem{H0CF}
J. Huang, 0-Hecke algebra actions on coinvariants and flags, J. Algebraic Combin. 40 (2014), 245--278.

\bibitem{H0SR}
J. Huang, 0-Hecke algebra action on the Stanley-Reisner ring of the Boolean algebra, Ann. Comb. 19 (2015), 293--323.

\bibitem{Hecke}
J. Huang, Hecke algebras with independent parameters, J. Algebraic Combin. 43 (2016), 521--551.

\bibitem{H0Tab}
J. Huang, A tableau approach to the representation theory of 0-Hecke algebras, Ann. Comb. {\bf 20} (2016), no.~4, 831--868. 

\bibitem{H0Ch}
J. Huang, A uniform generalization of some combinatorial Hopf algebras, Algebr. Represent. Theory {\bf 20} (2017), no.~2, 379--431. 

\bibitem{Humphreys}
J. E. Humphreys, {\it Reflection groups and Coxeter groups}, Cambridge Studies in Advanced Mathematics, 29, Cambridge University Press, Cambridge, 1990.

\bibitem{Konig}
S. K\"{o}nig, The decomposition of 0-Hecke modules associated to quasisymmetric Schur functions, Algebr. Comb. {\bf 2} (2019), no.~5, 735--751.

\bibitem{KrobThibon}
D. Krob\ and\ J.-Y. Thibon, Noncommutative symmetric functions. IV. Quantum linear groups and Hecke algebras at $q=0$, J. Algebraic Combin. {\bf 6} (1997), no.~4, 339--376. 

\bibitem{LiChen}
F. Li\ and\ L. Chen, The natural quiver of an Artinian algebra, Algebr. Represent. Theory {\bf 13} (2010), no.~5, 623--636. 

\bibitem{Lusztig81}
G. Lusztig, On a theorem of Benson and Curtis, J. Algebra {\bf 71} (1981), no.~2, 490--498.

\bibitem{Lusztig03}
G. Lusztig, {\it Hecke algebras with unequal parameters}, CRM Monograph Series, 18, American Mathematical Society, Providence, RI, 2003. 


\bibitem{Norton}
P. N. Norton, $0$-Hecke algebras, J. Austral. Math. Soc. Ser. A {\bf 27} (1979), no.~3, 337--357. 

\bibitem{OEIS}
N. J. A. Sloane, editor, \emph{The On-Line Encyclopedia of Integer Sequences}, published electronically at \url{https://oeis.org}, 2018.

\bibitem{Solomon}
L. Solomon, A decomposition of the group algebra of a finite Coxeter group, J. Algebra {\bf 9} (1968), 220--239. 

\bibitem{Monoid}
B. Steinberg, {\it Representation theory of finite monoids}, Universitext, Springer, Cham, 2016. 

\bibitem{Stembridge}
J. R. Stembridge, A short derivation of the M\"obius function for the Bruhat order, J. Algebraic Combin. {\bf 25} (2007), no.~2, 141--148. 

\bibitem{TvW15}
V. Tewari\ and\ S. van Willigenburg, Modules of the 0-Hecke algebra and quasisymmetric Schur functions, Adv. Math. {\bf 285} (2015), 1025--1065.

\bibitem{TvW19}
V. Tewari\ and\ S. van Willigenburg, Permuted composition tableaux, 0-Hecke algebra and labeled binary trees, J. Combin. Theory Ser. A {\bf 161} (2019), 420--452. 

\end{thebibliography}


\end{document}